\documentclass[a4paper, 11pt, reqno]{amsart}

\usepackage{amsmath, amsthm, amssymb, amsfonts,  hyperref, marginnote, color, holtpolt,fullpage}
\usepackage[utf8]{inputenc}
\usepackage[authoryear]{natbib}
\numberwithin{equation}{section}

\newcommand{\R}{\mathbb{R}}

\newtheorem{theorem}{Theorem}[section]
\newtheorem{cor}[theorem]{Corollary}
\newtheorem{lemma}[theorem]{Lemma}

\theoremstyle{remark}
\newtheorem{remark}[theorem]{Remark}

\newcommand{\lb}{\left(}
\newcommand{\rb}{\right)}
\newcommand{\lv}{\lvert}
\newcommand{\rv}{\rvert}
\renewcommand{\a}{\alpha}
\renewcommand{\b}{\beta}
\renewcommand{\l}{\lambda}
\renewcommand{\d}{\delta}
\newcommand{\s}{\sigma}

\newcommand{\floor}[1]{\lfloor #1 \rfloor}
\newcommand{\C}{\mathbb C}

\hypersetup{
    final,
    colorlinks=true,
    raiselinks=true, 
    linkcolor=[rgb]{.0, .0, .8}, 
    citecolor=[rgb]{.0, .0, .8}, 
    filecolor=black, 
    urlcolor=black 
}

\begin{document}
\title[Universality in Random Moment Problems]{Universality in Random Moment Problems}
\author[]{Holger Dette$^*$}
\address{$^*$Department of Mathematics, Ruhr University Bochum, 44780 Bochum,
Germany}
\email{holger.dette@ruhr-uni-bochum.de}

\author[]{Dominik Tomecki$^\dag$}
\address{$^\dag$Department of Mathematics, Ruhr University Bochum, 44780 Bochum,
Germany}
\email{dominik.tomecki@ruhr-uni-bochum.de}

\author{Martin Venker$^\ddag$}
   \address{$^\ddag$Faculty of Mathematics, Bielefeld University, 33501 Bielefeld, Germany}
   \email{mvenker@math.uni-bielefeld.de}

\keywords{Random moment sequences, universality, CLT, large deviations principles, Stieltjes transform, free probability.}
\subjclass[2010]{60F05, 30E05, 60B20}

\begin{abstract}
Let $\mathcal{M}_n(E)$ denote the set of vectors of the first $n$ moments of probability measures on $E\subset\R$ with existing moments. The investigation of such moment spaces in high dimension has found considerable interest in the recent literature.
For instance, it has been shown that a uniformly distributed moment sequence in $\mathcal M_n([0,1])$ converges in the large $n$ limit to the moment sequence of the arcsine distribution. In this article we provide  a  unifying viewpoint by identifying classes of more general distributions on $\mathcal{M}_n(E)$ for $E=[a,b],\,E=\R_+$ and $E=\R$, respectively, and discuss universality problems within these classes.
In particular, we demonstrate that the moment sequence of the arcsine distribution is not universal for $E$ being a compact interval. On the other hand, on the
moment spaces $\mathcal{M}_n(\R_+)$ and $\mathcal{M}_n(\R)$ the random moment
sequences governed by our distributions exhibit for $n\to\infty$ a universal behaviour: The first $k$ moments 
of such a random vector converge almost surely to  the first $k$ moments 
of the  Marchenko-Pastur distribution (half line) and Wigner's semi-circle distribution (real line). Moreover, the fluctuations around the limit sequences are Gaussian. 
We also obtain moderate and large deviations principles and discuss relations of our findings with free probability.
\end{abstract}

\maketitle

\section{Introduction } 
Let $\mathcal{P}(E)$ denote the set of probability measures on an (possibly infinite) 
 interval $E\subset \R$ with finite moments of all orders. For a measure  $\mu \in \mathcal{P}(E)$ denote by $m_j(\mu)=\int_E x^jd\mu(j)$ its $j$-th moment and define
\begin{align*}
\mathcal{M}_n(E): = \big \{(m_1(\mu),\dots,m_n(\mu)):\mu\in \mathcal {P}(E) \big \}
\end{align*}
as  the set of moment sequences up to order $n$, generated by $\mathcal {P}(E)$. The set $\mathcal{M}_n(E)$ is convex and has  
been the  subject of many studies beginning with
  \cite{karsha1953},  \cite{karlin1966} and  \cite{krenud1977}. In these classical works, geometric aspects of moment spaces were 
studied. While the even more classical moment problems deal with all \textit{possible} moment sequences, a probabilistic 
investigation rather asks how a \textit{typical} moment sequence looks like. This was initiated in  \cite{chakemstu1993}, where a  
uniform distribution on  $\mathcal{M}_n([0,1])$ was considered. There it was shown
that the first $k$ moments of such a random  vector $ (m_1^{(n)},\dots,m_n^{(n)}) $   in $\mathcal{M}_n([0,1])$ obey a law of large numbers,
  when $n$ tends to infinity (but $k$ is fixed), that is
  \begin{align}\label{lln}
 (m_1^{(n)},\dots,m_k^{(n)})  \xrightarrow{d}  (m_1^*,\dots,m_k^*)  ,\quad n\to\infty,
\end{align}
$\xrightarrow{d}$ denoting convergence in distribution.
Here $m_j^{(n)}$ is the $j$-th component of the random moment vector $(m_1^{(n)},\dots,m_n^{(n)})  $ and $m_j^*$ is the $j$-th moment of the arcsine distribution (on the interval $[0,1]$).
They also derived the central limit theorem
\begin{align}\label{CKS}
\sqrt{n}\big ( (m_1^{(n)},\dots,m_k^{(n)})  -(m_1^*,\dots,m_k^*)  \big  )\xrightarrow{d} \mathcal{N}(0,\Sigma_k ),\quad n\to\infty
\end{align}
 with the covariance matrix $\Sigma_k =(m_{i+j}^*-m_i^*m_j^*)^k_{i,j=1}$.
\cite{gamloz2004}  investigated corresponding large deviations principles, while  \cite{lozada2005}
 studied similar problems for moment spaces corresponding to more general functions defined on a bounded set. 
 
 More recently, \cite{detnag2012}   defined special probability distributions on the non-compact moment spaces  $\mathcal{M}_n ([0,\infty))$ and $\mathcal{M}_{2n-1}(\mathbb{R})$.
 They could establish results analogous to \eqref{CKS} with the moments of the arcsine distribution replaced by those of the Marchenko-Pastur distribution (on $[0,\infty)$) and of the semicircle distribution (on $\R$), respectively.

 In this article, we are going to investigate this surprising occurrence of the three distributions arcsine, Marchenko-Pastur and 
semicircle distribution in more detail. We are particularly interested in a possible universality of these distributions, as in 
random matrix theory the latter two  appear naturally for large classes of random matrices with independent entries (see e.g.~\cite{BaiSilverstein} and references therein). The arcsine measure also appears as a universal distribution of zeros of orthogonal polynomials with respect to 
weight functions on compact intervals (see \cite{StahlTotik}).
Especially for unbounded moment spaces a clarification of universality seems desirable, as there is no uniform measure and thus 
the consideration of a particular probability measure needs justification. In other words, we are asking for how typical the moment sequences of arcsine, semicircle and Marchenko-Pastur distribution are.

 The paper will be organized as follows. In Section \ref{sec2} we review some basic facts about moment spaces and introduce 
general classes of distributions on the moment spaces under consideration. They keep two key features of 
the uniform distribution on $\mathcal{M}_n([a,b])$ and can be used to interpolate between distributions on compact and non-compact 
moment spaces. For these distributions we derive laws of large numbers of the type \eqref{lln}. In particular, we show that for 
moment spaces $\mathcal{M}_n([a,b])$ corresponding to compact intervals there is no universality of the arcsine 
distribution. Instead, the arising measures are known as free binomial distributions, i.e.~the analogues of the binomial distribution 
in free probability theory. On the other hand, for the moment spaces $\mathcal{M}_n([0,\infty))$ and 
$\mathcal{M}_{n}(\mathbb{R})$ the first $k$ moments of a random vector always converge to the first $k$ moments of 
Marchenko-Pastur 
and semicircle distributions, respectively. The occurrence of both distributions will be explained in terms of free Poissonian 
and free central limit theorems for the free binomial distribution. 
 In Section  \ref{sec3} we consider  central limit theorems of the form \eqref{CKS} and investigate moderate and large deviations principles for random moment sequences. 
 All proofs are postponed to Section \ref{sec4}.
 Our results provide an extensive description of the distributional properties of random moment sequences and a unifying view on several findings in the recent literature.  

 \section{Laws of Large Numbers} \label{sec2}
To motivate the class of distributions considered in this paper, we remark first that
a real valued sequence $(m_i)_{i \in \mathbb{N}_0}$ is a sequence of moments corresponding to a Borel measure on the real line if and only if all Hankel matrices  $(m_{i+j})^n_{i,j=0}$ are positive semi-definite (see \cite{hamburger1920}). Similar characterizations exist for measures supported on the half line  $[0,\infty)$ and  compact intervals, and the corresponding sequences are called Stieltjes and Hausdorff moment sequences  (see \cite{dettstud1997}). Due
   to   restrictions and relations of this type, the components of a random  moment vector in $\mathcal{M}_n(E)$ 
   are generically not independent coordinates.  
   Moreover, for a compact interval  $E$ the moment space $\mathcal{M}_n(E)$ is a rather small set. For instance, it is known that 
the volume of $\mathcal{M}_n([0,1])$ is of order $\mathcal O(2^{-n^2})$ (see  \cite{karsha1953}), as for a given moment sequence 
$(m_1,\dots,m_{n-1})\in \mathcal{M}_{n-1}([0,1])$, the possible range of the $n$-th moment   $m_n$ is very small.

   For these reasons, we will consider different sets of coordinates that  scale with the possible range of values. Although there are infinitely many choices of such coordinates, some are particularly natural and have found considerable attention in the literature. To be precise, assume that $(m_1,\dots,m_{j-1}) \in \mathcal{M}_{j-1} ([a, b])$ is a given vector of moments up to the order $j-1$.
    Then, because of convexity of $ \mathcal{M}_{j} ([a, b])$, the set of possible values $m_{j}$ 
    $$
    \big \{ m_j(\mu) \big |  ~\mu \in \mathcal{P}([a, b]) ;  ~m_i(\mu)  =  m_i ~\text{ for all } i=1,\ldots ,j-1 \big\} 
    $$
    is a compact interval, say $[m_j^-,m_j^+]$. Following \cite{dettstud1997}, we define for $m_j^+\not=m_j^-$ and a given $j$-th moment $m_j$ the $j$-th canonical moment $p_j$ via
\begin{align*}
p_j:=\frac{m_j-m_j^-}{m_j^+-m_j^-}.
\end{align*}
The canonical moments are left undefined if $m^-_j = m^+_j$ (in this case the vector $(m_1,\ldots, m_{j-1})$ is a boundary point of 
the set $\mathcal{M}_{j-1}([a, b])$ - see \cite{karlin1966}).
Clearly, $p_j\in [0,1]$, and  $p_j$ gives the relative position of $m_j$ in the available section of the set $\mathcal{M}_j([a, b])$. It is also worthwhile to mention that canonical moments are invariant under linear transformations of the measure (see \cite{dettstud1997}, p. 13).
The correspondence map  
\begin{equation} \label{phimap}
\varphi_n^{[a,b]}:  {\vec p_n} =(p_1,\dots ,p_n) \mapsto { \vec  m_n} =(m_1,\dots ,m_n)
\end{equation} 
 between the  canonical and ordinary moments is one-to-one from $(0,1)^n$ onto ${ \rm Int }(\mathcal{M}_n ([a, b]))$ 
($\textup{Int}$ denoting the interior)
 and many classical quantities of the measure, especially of its associated orthogonal polynomials and the continued fraction expansion of its Stieltjes transform, have expressions in terms of the canonical moments (see \cite{dettstud1997} for more details).
Canonical moments were introduced in a series of papers by  \cite{skibinsky1967,skibinsky1968,skibinsky1969} and 
are  closely related to the Verblunsky coefficients, which were investigated much earlier by
 \cite{verblunsky1935,verblunsky1936} for measures on the unit circle.

In case of the uniform distribution on $\mathcal M_n([0,1])$, as studied in \cite{chakemstu1993}, the canonical moments 
have two important properties. After a change of variables by \eqref{phimap}, the uniform 
 distribution on $\mathcal{M}_n([0,1])$ has a density w.r.t.~the Lebesgue measure on $(0,1)^n$ proportional to
 \begin{align}
 \prod_{j=1}^n (p_j(1-p_j))^{ n-j } = 
 \exp\Big[\sum_{j=1}^n (n-j)\log(p_j(1-p_j)) \Big].\label{def_uniform_canonical}
 \end{align}
 Thus, the canonical moments are independent and for $n\gg j$ nearly 
identically distributed. To investigate a possible universality of the arcsine distribution, we will now define a class of 
distributions respecting these two properties. However, 
we will generalize the situation by allowing for different distributions of even and odd canonical moments. This takes into account the different roles that 
even and odd moments 
play. While even moments are always positive and give some rough information about the size of the
support of the measure, odd moments give information about location of the support and the symmetry of the measure. In canonical moments, symmetry around the center of $[a,b]$ 
can 
be characterized easily as the property that all odd canonical moments are $1/2$ (see \cite{skibinsky1969}).

Let $V_1,V_2:{[0,1]}\rightarrow{\R}$ be continuous functions. Define the probability measure 
$\mathbb{P}_{n,[a,b],V_{1,2}}$ on $\mathcal{M}_n ([a,b])$ by $\mathbb{P}_{n,[a,b],V_{1,2}}\lb\partial\mathcal{M}_n ([a,b])\rb=0$ and on ${\rm Int} ( \mathcal{M}_n ([a,b]))$  via the density 
\begin{align}
P_{n,[a,b],V_{1,2}}(m_1,\dots,m_n):=\frac{1}{Z_{n,[a,b],V_{1,2}}}\exp\Big[-n \sum_{j=1}^{\floor{\frac{n+1}2}} V_1(p_{2j-1})-n 
\sum_{j=1}^{\floor{\frac n2}} V_2(p_{2j})\Big]\label{def_compact}
\end{align}
w.r.t.~the $n$-dimensional Lebesgue measure, where $p_j=p_j(m_1,\dots,m_j)$ is the $j$-th canonical moment 
of the sequence $(m_1,\dots,m_n) \in {\rm Int } ( \mathcal{M}_n ([a,b]) )$
defined by \eqref{phimap}  $(j=1,\ldots,n)$ and $Z_{n,[a,b],V_{1,2}}$ is the normalization constant. By $\floor{x}$ we denote the 
largest natural number smaller or equal to $x$.
Note that the case $V_1(x)=V_2(x)\equiv 0$ and  $[a,b]=[0,1]$ has been considered in \cite{chakemstu1993}. The factors $n$ in the 
exponent in \eqref{def_compact} are asymptotically equivalent to the factor $n-j$ in \eqref{def_uniform_canonical}. It 
follows from \eqref{def_uniform_canonical} that under $\mathbb{P}_{n,[a,b],V_{1,2}}$ the odd, respectively even, canonical moments 
are nearly i.i.d..

Let us 
now  formulate our first result for random moment sequences on measures supported on the interval $[a,b]$. Here and later on, we 
will tacitly assume that the random variables $(m_j^{(n)})_{j,n\geq1}$ are defined on the same probability space.

\begin{theorem}\label{thrm_compact_1}\strut
\begin{enumerate}
\item Let $a<b$ and $V_1,V_2\in C^2((0,1))$ be continuous at 0 and 1. Assume that the functions
\begin{align*}
W_1(p) := V_1(p)-\log(p(1-p))\qquad \text{ and }\qquad W_2(p) := V_2(p)-\log(p(1-p))
\end{align*}
each have a unique minimizer $p_1^*\in(0,1)$ and $p_2^*\in(0,1)$, respectively. Let $m^{(n)}=(m_1^{(n)},\dots,m_n^{(n)})$ be drawn 
from $\mathbb{P}_{n,[a,b],V_{1,2}}$ and abbreviate $q^*_{i}:=1-p^*_{i},i=1,2$. Then we have for each $k\geq1$ as $n\to\infty$
\begin{align*}
(m_1^{(n)},\dots,m_k^{(n)})\to(m^*_1,\dots,m^*_k)
\end{align*}
almost surely and in $L^1$, where $m^*_1,\dots,m^*_k$ are the first $k$ moments of a probability measure 
$\mu_{p^*_{1},p^*_2}=\mu_{p^*_{1},p^*_2}^{ac}+\mu_{p^*_{1},p^*_2}^{d}.$ Setting
\begin{align*}
l_\pm:=a+(b-a)\lb\sqrt{p_1^*q_2^*}\pm\sqrt{p_2^*q_1^*}\rb^2,
\end{align*}
 the measures $\mu_{p^*_{1},p^*_2}^{ac}$ and $\mu_{p^*_{1},p^*_2}^{d}$ are given by
\begin{align*} 
&\mu_{p^*_{1},p^*_2}^{ac}(dx)=\frac{\sqrt{(x-l_-)(l_+-x)}}{2\pi p_2^*(x-a)(b-x)}1_{[l_-,l_+]}(x)dx,\\
&\mu_{p^*_{1},p^*_2}^{d}=\lb 1-\frac{p_1^*}{p_2^*}\rb_+\d_a+\lb\frac{p_1^*+p_2^*-1}{p_2^*}\rb_+\d_b.
\end{align*}
Here $(y)_+$ denotes the positive part of $y \in \mathbb{R}$  and $\d_y$ is the Dirac measure at the point $y$.
\item
If $p_1^*,p_2^*$ are such that $\mu_{p^*_{1},p^*_2}$ does not have atoms, then $\mu_{p^*_{1},p^*_2}$ is the equilibrium measure on 
the interval $[a,b]$ to the external field
\begin{align*}
Q(t):={}&-\lb\frac{p_1^*}{p_2^*}-1\rb\log(t-a)-\lb\frac{1-p_1^*-p_2^*}{p_2^*}\rb\log(b-t), 
\end{align*}
i.e.~$\mu_{p^*_{1},p^*_2}$ is the unique Borel probability measure on the interval $[a,b]$ minimizing the functional
\begin{align}
\mu\mapsto\int_a^bQ(t)d\mu(t)-\int_a^b\int_a^b\log\lvert t-s\rvert d\mu(t)d\mu(s).\label{eq_measure}
\end{align}
\end{enumerate}
\end{theorem}

\vspace{4em}

\begin{remark} ~~
\begin{enumerate}
	\item
If $p_1^*=p_2^*=1/2$, the measure $\mu_{p^*_1,p_2^*}$ in Theorem~\ref{thrm_compact_1}  is the arcsine distribution on the interval   $[a,b]$. Note that this does not imply $V_1=V_2\equiv 0$. However, we see that for $p^*_1\not=1/2$ or $p_2^*\not=1/2$, the limiting measure (the measure having the limiting moments) is not the arcsine measure or an affine rescaling of it. We conclude that the moments of the arcsine measure are not universal within the class of random moment sequences in $\mathcal{M}_n([a,b])$ with nearly i.i.d.~canonical moments. On the other hand, there is still some universality 
as the limiting measure only depends on $V_1,V_2$ via the parameters $p_1^*$ and $p_2^*$.
\item Since for probability measures supported on a fixed compact set convergence of moments is equivalent to convergence in 
distribution, the convergence result of Theorem \ref{thrm_compact_1} can be restated as follows: Let $\mu_n\in \mathcal P([a,b])$ be 
a random probability measure with first $n$ moments $(m_1^{(n)},\dots,m_n^{(n)})$ which are $\mathbb 
P_{n,[a,b],V_{1,2}}$-distributed. Then $\mu_n$ converges a.s.~(and in expectation) weakly to $\mu_{p_1^*,p_2^*}$ as $n\to\infty$. 
\end{enumerate}
\end{remark}
The measure $\mu_{p_1^*,p_2^*}$ is known in the literature under (at least) two different names. In the context of probability 
theory on graphs, it is called Kesten-McKay measure (see \cite{Kesten,McKay}). It has also been studied in the context of orthogonal 
polynomials (see \cite{CohenTrenholme,SaitohYoshida,CastroGrunbaum}). In free probability, it is called free binomial distribution 
(see \cite{nicspe2006}). It will turn out useful to explain this naming in more detail. 

Free probability is a variant of non-commutative probability theory initiated by Voiculescu (see \cite{nicspe2006} or 
Chapter 22 by Speicher in \cite{Handbook} for an introduction and references) that has found its applications in particular in random matrix 
theory. For our purposes it suffices to know that free probability theory uses a different notion of independence, called freeness, 
that manifests itself in a different convolution of probability measures. A constructive approach to this convolution uses random 
matrices: Let $H_{1,n},H_{2,n}$ be deterministic diagonal $n\times n$ matrices with diagonal entries $h_{1,n}(ii)$ and 
$h_{2,n}(ii)$, 
respectively. Assume that the empirical measures of the diagonal entries, i.e.~the eigenvalues, converge for $n\to\infty$ weakly to 
probability measures of bounded support $\mu_1$ and $\mu_2$, respectively, that is
\begin{align*}
\lim_{n\to\infty}\frac1n\sum_{i=1}^n\d_{h_{j,n}(ii)}=\mu_j, \quad j=1,2,\quad \text{ weakly.}
\end{align*}
Now let for each $n$ a Haar distributed random unitary $n\times n$ matrix $U_n$ be given on a common probability space. The Haar 
probability measure on the unitary group $\mathcal U_n$ is the unique Borel probability measure that is invariant under left (and 
right) multiplication with any group element. Letting $x_1,\dots,x_n$ denote the $n$ real random eigenvalues of the Hermitian 
random matrix $H_{1,n}+U_nH_{2,n}U_n^*$, the empirical measure of the $x_i$'s converges for $n\to\infty$ almost surely in 
distribution to a non-random limit. This limit is called the free (additive) convolution of $\mu_1\boxplus\mu_2$, in symbols
\begin{align*}
\mu_1\boxplus\mu_2:=\lim_{n\to\infty}\frac1n\sum_{i=1}^n\d_{x_i}\qquad \text{ a.s.~weakly}.
\end{align*}
In analogy to classical probability, the free binomial distribution with parameters $n\in\mathbb N$ and $p\in[0,1]$ is then the 
$n$-fold free convolution of the Bernoulli distribution $\mu=(1-p)\d_0+p\d_1$ with itself. It seems convenient to extend the name to 
convolutions of measures $\mu=(1-p)\d_c+p\d_d$ with itself, $c,d\in\R$. Moreover, even fractional convolution numbers are possible 
using an analytic approach to the free convolution via the so-called $R$-transform (see \cite[Chapter 22]{Handbook}). It seems 
difficult to give a direct interpretation of the occurence of the free binomial distribution in the context of random moments. For 
instance it is not hard to verify that for $\mu=\frac12\d_c+\frac12\d_d$ the free convolution $\mu\boxplus\mu$ is the arcsine 
measure with support $[c+d-\sqrt{c^2+d^2},c+d+\sqrt{c^2+d^2}]$, but in general the measure $\mu_{p_1^*,p_2^*}$ is not just a 
two-fold convolution of a Bernoulli measure with itself.

However, free probability indicates that universal limiting measures may be expected if random moment problems are considered for the moment spaces $\mathcal{M}_n (\mathbb{R}_+)$ with $\R_+:=[0,\infty)$ and $\mathcal{M}_n(\mathbb{R})$. Indeed, analogous to classical probability, there are free analogs of Poisson limit theorem and central limit theorem for the free binomial distribution \cite[Chapter 22]{Handbook}. Typically, they are considered for $\mu=(1-p_m)\d_0+p_m\d_1$ and show weak convergence of the rescaled $n$-th convolution power $\mu^{\boxplus m}$ to the free Poisson (Marchenko-Pastur distribution) or the free Gaussian law (semicircle distribution), as $m\to\infty$ and $p_m$ converges to a zero or non-zero number, respectively. 

The following corollary can be seen as a variant of these limit theorems. The proof is straightforward and will be omitted.

\begin{cor}
	Let for each $m\in\mathbb N$ $a_m<b_m$	and $p^*_{1,m},p_{2,m}^*\in(0,1)$ be given.
\begin{enumerate}
\item Assume that, as $m\to\infty$,
\begin{align*}
&a_m\to0,\ b_m\to\infty,\quad p^*_{1,m},p_{2,m}^*\to0\ \text{ such that }\\
 &p^*_{i,m}b_m\to z_i^*, i=1,2,
\end{align*}
for some constants $z_1^*,z_2^*>0$. Then the measure $\mu_{p^*_{1,m},p_{2,m}^*}$ defined in Theorem \ref{thrm_compact_1} 
on the interval $[a_m,b_m]$ converges in the large $m$ limit weakly to the measure $\mu_{MP,z_1^*,z_2^*}$, where with 
$l_\pm:=(\sqrt{z_1^*}\pm\sqrt{z_2^*})^2$
\begin{align}
&\mu_{MP,z_1^*,z_2^*}(dx)=\lb 1-\frac{z_1^*}{z_2^*}\rb_+\d_0+\frac{1}{2\pi z_2^*}\frac{\sqrt{(x-l_-)(l_+-x)}}{x}1_{[l_-,l_+]}(x)dx.\label{MP}
\end{align}
The density of the absolutely continuous part of $\mu_{p^*_{1,m},p_{2,m}^*}(x)$ converges pointwise to the density of the absolutely continuous part of $\mu_{MP,z_1^*,z_2^*}$ and uniformly within compact subsets of $(l_-,l_+)$.  Moreover, the moments of $\mu_{p^*_{1,m},p_{2,m}^*}$ converge to the moments of $\mu_{MP,z_1^*,z_2^*}$.
\item  Assume that, as $m\to\infty$,
\begin{align*}
&a_m\to -\infty,\ b_m\to\infty,\\ 
&p^*_{2,m}\lv a_m\rv b_m\to \b^*,\quad a_m+(b_m-a_m)p^*_{1,m}\to\a^*
\end{align*}
for constants $\a^*\in\R,\b^*>0$. Then the measure $\mu_{p^*_{1,m},p^*_{2,m}}$ defined in Theorem \ref{thrm_compact_1} 
on the interval $[a_m,b_m]$ 
converges weakly in the large $m$ limit to the measure $\mu_{SC,\a^*,\b^*}$, where with $l_\pm:=\a^*\pm2\sqrt{\b^*}$
 \begin{align}
\mu_{SC,\a^*,\b^*}(dx)=\frac{1}{2\pi\b^*}\sqrt{(x-l_-)(l_+-x)}1_{[l_-,l_+]}(x)dx.\label{SC}
 \end{align}
The density of the absolutely continuous part of $\mu_{p^*_{1,m},p_{2,m}^*}(x)$ converges pointwise to the density of  $\mu_{SC,\a^*,\b^*}$ and uniformly within compact subsets of $(l_-,l_+)$.  
 Moreover, the moments of $\mu_{p^*_{1,m},p^*_{2,m}}$ converge to the moments of $\mu_{SC,\a^*,\b^*}$.
\end{enumerate}	
\end{cor}

\vspace{4em}

\begin{remark} \strut
\begin{enumerate}
\item
The measure $\mu_{MP,z_1^*,z_2^*}$ is called Marchenko-Pastur distribution  (see \cite{hiaipetz2000} or \cite{nicspe2006}). For $z_1^*\geq z_2^*$ (absolutely continuous case) it is the 
equilibrium measure on $\R_+$ (in the sense of \eqref{eq_measure}) to the field 
    \begin{align*}
        Q(t)=\frac t{z_2^*}-\frac{z_1^*-z_2^*}{z_2^*}\log t.
    \end{align*}
				Besides its role in free probability theory as the free analog of the Poisson distribution it is 
particularly well-known for its universality in random matrix theory. More precisely, let $X$ denote an $m\times n$ random matrix 
with real i.i.d.~entries having mean 0 and variance $\s^2>0$. Assume that as $m,n\to\infty$ we have $m/n\to\l\in(0,\infty)$. Then 
the empirical distribution of the eigenvalues of the sample covariance matrix $XX^T/n$ converges a.s.~and in expectation 
weakly to $\mu_{MP,z_1,z_2}$, where $z_1:=\s^2(1+\sqrt{\l})/(1+\sqrt\l)^2$ and $z_2:=\l z_1$. For this result and generalizations we 
refer to \cite{BaiSilverstein} and references therein.
				
 \item  The measure $\mu_{SC,\a^*,\b^*}$ is called semicircle distribution. It is the equilibrium measure to the field 
    \begin{align*}
        Q(t)=\frac{t^2}{2\b^*}-\frac{\a^*t}{\b^*}.
    \end{align*}
In free probability, it plays the role of the Gaussian distribution. In random matrix theory it is the universal limit of so-called 
Wigner matrices: Let $X$ be an $n\times n$ random matrix with real i.i.d.~mean 0 and variance $\s^2>0$ entries on and above the 
diagonal and the entries below the diagonal are chosen such that $X$ is symmetric. Then the empirical distribution of the 
eigenvalues of $X/\sqrt{n}$ converges a.s.~and in expectation weakly to $\mu_{SC,\a,\b}$
as  $n\to\infty$, where $\a=0$ and $\b=\s^2$, see e.g.~\cite{BaiSilverstein}.
	
		The universality in these random matrix statements lies in the fact that the limiting distribution is always the same regardless of the distribution of the matrix entries.
		\item The measures $\mu_{p_1^*,p_2^*},\mu_{MP,z_1^*,z_2^*}$ and $\mu_{SC,\a^*,\b^*}$ all belong to the so-called free Meixner class. It consists of the free analogues of the six classical Meixner class distributions which are  Gaussian, Poisson, gamma, binomial, negative binomial and hyperbolic secant distribution. The distributions of the free Meixner class enjoy some interesting characterizing properties, for instance having a generating function of resolvent type for the corresponding orthogonal polynomials (see \cite{Anshelevich} for details) in analogy to the generating functions of the classical Meixner class being of exponential type (see \cite{Meixner}).
		\end{enumerate}
\end{remark}

Let us now turn to infinite moment spaces, starting with $\mathcal{M}_n(\R_+)$ (recall $\R_+=[0,\infty)$). Following \cite{detnag2012}, we may define the canonical moments $z_1,\dots,z_n$ of a moment sequence $m_1,\dots,m_n$ in the interior of $M_n(\R_+)$ as
\begin{align*}
z_k:=\frac{m_k-m_k^-}{m_{k-1}-m_{k-1}^-},\quad k=1,\dots,n,
\end{align*}
$m_0^-=0, m_0=1$.
Here one uses that given $m_1,\dots,m_{k-1}$, the section of possible values of $m_k$
for given moments $(m_1,\ldots,m_{k-1}) \in {\rm Int} (\mathcal{M}_{k-1} (\mathbb{R}_+))$
 is an interval of the form $[m_k^-,\infty)$ (see \cite{karlin1966}, Chapter V). Clearly, $z_k\in\R_+$. 
 The correspondence 
 \begin{equation} \label{phimaphalf}
\varphi_n ^{\mathbb{R}_+}:  {\vec z_n} =(z_1,\dots ,z_n) \mapsto { \vec  m_n} =(m_1,\dots ,m_n) 
\end{equation} 
between canonical and ordinary moments is one-to-one from $(0,\infty)^n$ onto  ${\rm Int} (\mathcal{M}_n(\mathbb{R}_+))$ (for all $n \in \mathbb{N}$). The Jacobian of this transformation is readily computed as
\begin{align}\label{Jacobian_R_+}
\left\lv\prod_{k=1}^n\frac{\partial m_k}{\partial z_k}\right\rvert=\prod_{k=1}^n(m_{k-1}-m_{k-1}^-)=\prod_{k=2}^nz_1z_2\dots z_{k-1}=\prod_{k=1}^n z_k^{n-k} ~.
\end{align}
To define a probability measure on $ {\rm Int} ( \mathcal{M}_n(\R_+) )$, consider continuous functions $V_1,V_2:{\R_+}\rightarrow{\R}$, such that for some $\varepsilon>0$ and all $z$ large enough the inequality
\begin{align}
\frac{V_i(z)}{\log z}\geq 2+\varepsilon\label{growth_condition},\ i=1,2
\end{align}
holds. Then define  a probability measure $\mathbb{P}_{n,\R_+,V_{1,2}}$ on $\mathcal{M}_n(\R_+)$ by $\mathbb{P}_{n,\R_+,V_{1,2}}\lb\partial\mathcal{M}_n (\R_+)\rb=0$ and on $ {\rm Int} ( \mathcal{M}_n(\R_+) )$
via  the density
\begin{align}
P_{n,\R_+,V_{1,2}}(m_1,\dots,m_n):=\frac{1}{Z_{n,\R_+,V_{1,2}}}\exp\Big[-n \sum_{j=1}^{\floor{\frac{n+1}2}} V_1(z_j)-n \sum_{j=1}^{\floor{\frac{n}2}} V_2(z_j)\Big],\label{def_R_+}
\end{align}
where 
$Z_{n,\R_+,V_{1,2}}$ is the normalizing constant such that $P_{n,\R_+,V_{1,2}}$ is a probability density   with respect to the Lebesgue measure on ${\rm Int} ( \mathcal{M}_n(\R_+)) $. This is possible due to \eqref{Jacobian_R_+} and \eqref{growth_condition}. Because of \eqref{Jacobian_R_+}, the canonical moments $z_1,z_2,\dots,z_k$ are independent under $\mathbb{P}_{n,\R_+,V_{1,2}}$ and for large $n$ and fixed $k$
nearly identically distributed. 

Note that \cite{detnag2012} considered the special case of \eqref{def_R_+} with $V_1(t)=V_2(t)=t-\frac{c}{n}\log t$ and showed that under this measure the (ordinary) moments converge to those of the Marchenko-Pastur distribution. Here we will show that the moments of the Marchenko-Pastur distribution are in fact universal for all generic functions $V_1,V_2$.

\begin{theorem}\label{thrm_R_+}
	Let $V_1,V_2\in C^2((0,\infty))$ be continuous at $0$, satisfy \eqref{growth_condition} and assume that
	\begin{align*}
	W_1(z) := V_1(z)-\log z\qquad \text{ and }\qquad W_2(z) := V_2(z)-\log z
\end{align*}
each have a unique minimizer $z_1^*\in(0,\infty)$ and $z_2^*\in(0,\infty)$, respectively. Let the vector $m^{(n)}=(m_1^{(n)},\dots,m_n^{(n)})$ be drawn from $\mathbb{P}_{n,\R_+,V_{1,2}}$. Then we have for any $k\geq1$ as $n\to\infty$
	\begin{align*}
	(m_1^{(n)},\dots,m_k^{(n)})\to(m^*_1,\dots,m^*_k)
	\end{align*}
	almost surely and in $L^1$, where $m^*_1,\dots,m^*_k$ are the first $k$ moments of the Marchenko-Pastur distribution $\mu_{MP,z_1^*,z_2^*}$ defined in \eqref{MP}, that is
\begin{align}
    m_j^*= \sum \limits_{i = 0}^{\lfloor \frac{j - 1}{2} \rfloor} \binom{j - 1}{i}(z_1^*)^{i + 1}(z_2^*)^i (z_1^* + z_2^*)^{j - 1 - i} \frac{1}{i + 1} \binom{2i}{i}. \nonumber 
\end{align}

\end{theorem}
Next, we consider the  moment space corresponding to measures supported on $\R$. We will use the recurrence coefficients of the corresponding orthogonal polynomials as a coordinate system.  To be precise, note that for any measure $\mu\in \mathcal{P}(\R)$ there is a sequence of monic polynomials $P_0(x),P_1(x),\dots$ with $\operatorname{deg} P_j=j$ that is orthogonal in $L^2(\mu)$. If $\mu$ is supported on finitely many points, the sequence is finite. In any case, $P_j(x)$ depends on the measure $\mu$ via its moment sequence $(m_1,\dots,m_{2j-1})$ only. The orthogonal polynomials satisfy a three-term recurrence relation of the form
\begin{align}\label{recurrence}
&P_{j+1}(x)=(x-\a_{j+1})P_j(x)-\b_jP_{j-1}(x),\qquad j=1,\dots\\
&P_0(x)=1,\quad P_1(x)=x-\a_1\notag
\end{align}
with recurrence coefficients $\a_1,\a_2,\dots\in\R$ and $\b_1,\b_2,\dots>0$. 
For more  details regarding orthogonal polynomials we refer to \cite{chihara1978}. 
The mapping 
\begin{equation} \label{phimapreal}
\varphi_{2n-1} ^{\mathbb{R}}:  (\a_1,\b_1, \a_2, \dots,\b_{n - 1}, \a_n) \mapsto { \vec  m_{2n-1}} =(m_1,\dots ,m_{2n-1})
\end{equation}
is one-to-one from $(\mathbb{R} \times (0, \infty))^{n - 1} \times \mathbb{R}$ onto $\operatorname{Int}(\mathcal{M}_{2n-1}(\mathbb{R}))$ 
(for all $n \in \mathbb{N}$).
 Moreover, as observed by \cite{detnag2012},
 $(\a_1,\b_1,\a_2,\dots,\b_{n-1},\a_n)$ constitutes a system of independent coordinates on the moment space 
$\mathcal{M}_{2n-1}(\R)$. The corresponding Jacobian is given by
\begin{align*}
    \det D\varphi_{2n-1}^\mathbb{R} = \prod \limits_{j = 1}^{n - 1} \beta_j^{2n-2j-1}.
\end{align*}
Similarly, we may define a map for moment spaces of even order.
\begin{lemma}\label{phi_real_even}
    There is a bijection
    \begin{align}
            \varphi_{2n}^\mathbb{R}\ :\ (\mathbb{R} \times (0, \infty))^n  & \to  
\operatorname{Int}(\mathcal{M}_{2n}(\mathbb{R})),\notag\\
            (\alpha_1, \beta_1, \alpha_2, \ldots, \alpha_n, \beta_n)  & \mapsto    (m_1, \ldots, m_{2n})
        \label{phimapreal2}
    \end{align}
    between the recursion coefficients of the orthogonal polynomials and the corresponding moments. The Jacobian of 
$\varphi_{2n}^{\mathbb{R}}$ is
    \begin{align*}
        \det D\varphi_{2n}^\mathbb{R} = \prod \limits_{j = 1}^{n - 1} \beta_j^{2n - 2j}.
    \end{align*}
\end{lemma}

The values $\b_j$ have a simple interpretation in terms of moments, as
\begin{align*}
\b_j=\frac{m_{2j}-m_{2j}^-}{m_{2j-2}-m_{2j-2}^-},\qquad j=1,\dots,n,
\end{align*}
is the ratio of two consecutive even moments. The coefficients $\a_j$ give information about symmetry of the measure, e.g.~for $\mu$ symmetric around 0, one has $\a_j=0$ for all $j$. Taking into account these two different roles, we will again consider two continuous functions $V_1:{\R}\rightarrow{\R}$ and $V_2:{\R_+}\rightarrow{\R}$ such that for some $\varepsilon>0$ and $\lv \a\rv,\b$ large enough
\begin{align}\label{nbdg}
\frac{V_1(\a)}{\log \lv\a\rv}\geq 1+\varepsilon,\qquad \frac{V_2(\b)}{\log \b}\geq 3+\varepsilon.
\end{align}
With these notations we define the probability measure $\mathbb{P}_{n,\R,V_{1,2}}$   on $\mathcal{M}_{n}(\R)$ by\\ $\mathbb{P}_{n,\R,V_{1,2}}\lb\partial\mathcal{M}_{n} (\R)\rb=0$ and on $ {\rm Int} ( \mathcal{M}_{n}(\R) )$
via  the density
\begin{align*}
P_{n,\R,V_{1,2}}(m_1,\dots,m_{n}):=\frac{1}{Z_{n,\R,V_{1,2}}}\exp\Big[-n \sum_{j=1}^{\lfloor \frac{n + 1}{2} \rfloor} V_1(\a_j)-n 
\sum_{j=1}^{\lfloor \frac{n}{2} \rfloor}V_2(\b_j)\Big],
\end{align*}
and obtain the   following universal law of large numbers.
\begin{theorem}\label{thrm_R}
	Let $V_1\in C^2(\R),V_2\in C^2((0,\infty))$ be continuous at 0 and satisfy \eqref{nbdg}. Furthermore, assume that
	\begin{align*}
	W_1(\a) := V_1(\a)\qquad \text{ and }\qquad W_2(\b) := V_2(\b)-2\log\b
	\end{align*}
	each have unique minimizers $\a^*\in\R$ and $\b^*\in(0,\infty)$, respectively. Let $m^{(n)}=(m_1^{(n)},\dots,m_{n}^{(n)})$ 
be drawn from $\mathbb{P}_{n,\R,V_{1,2}}$. Then for any $k\geq1$ as $n\to\infty$
	\begin{align*}
	(m_1^{(n)},\dots,m_k^{(n)})\to(m^*_1,\dots,m^*_k)
	\end{align*}
	almost surely and in $L^1$, where $m^*_1,\dots,m^*_k$ are the first $k$ moments of the semicircle distribution $\mu_{SC,\a^*,\b^*}$ defined in \eqref{SC},
	that is
    \begin{align}
        m_j^*= \sum \limits_{i = 0}^{\lfloor j/2\rfloor}  \binom{j}{2i}(\beta^*)^{i} (\alpha^*)^{j - 2i}\frac{1}{i + 1} \binom{2i}{i} .\label{moments_SC}
    \end{align}
\end{theorem}
We finish this section with some concluding remarks concerning the class of models we consider. We study random moment sequences with independent and nearly identically distributed canonical moments or recurrence coefficients, respectively. Dropping either of the two properties will in general result in non-universal limiting sequences even on unbounded intervals, if there is any limit at all. Nevertheless, other related models have been used for successful studies of random matrix models. More precisely, so-called Gaussian beta ensembles admit tridiagonal matrix models, see \cite{DE}. More recently, \cite{KRV} have used tridiagonal matrix models for studying non-Gaussian beta ensembles. They consider $\exp(-n\operatorname{Tr}Q(T))\det(D\varphi_n^\R)$ as density on the space of recursion coefficients, where $T$ is the symmetric tridiagonal matrix (truncated Jacobi operator) with the $\a_j$'s on the main diagonal and $\b_j$'s on the neighboring diagonals, $Q$ is a strictly convex polynomial and $\operatorname{Tr}$ denotes the trace. It is not hard to see from the results in \cite{KRV} that the limiting moments corresponding to this model are those of the equilibrium measure to $Q$  (see \eqref{eq_measure}), only for $Q$ quadratic (this case is the one studied in \cite{DE}) the moments of the semicircle appear.

The connection between certain random matrix ensembles and canonical moments/recursion coefficients has also been used in \cite{GNR1} and \cite{GNR2} for deriving so-called sum rules for free binomial, semicircle and Marchenko-Pastur distribution.

\section{Asymptotic Normality, Moderate and Large Deviations} \label{sec3}
In this section, we examine 
the fluctuations of the random moment sequences  around their non-random limits. We state the central limit theorem and moderate and large deviations
results.  For the uniform distribution on the moment
space  $\mathcal{M}_n ([0,1])$, results of this type were obtained by   \cite{chakemstu1993} and
 \cite{gamloz2004}, respectively.  
 The following theorem shows that the fluctuations of  random moment vectors around their limits are Gaussian. We will adopt a 
short notation that allows us to state the three cases $E=[a,b]$, $E=\R_+$, $E=\R$ simultaneously. Note that the functions $W_1,W_2$ as well as the limiting moments $m_j^*$ differ, depending on $E$.
\begin{theorem}\label{thrm_CLT}\strut
In the situation of Theorem \ref{thrm_compact_1}, Theorem~\ref{thrm_R_+} or Theorem~\ref{thrm_R}, assume that $W_i''(y_i^*) \ne 
0$ for $i = 1, 2$, where 
\begin{align*}
 y_i^*:=\begin{cases}
         p_i^*&, \text{ if }E=[a,b],\\
	 z_i^*&,\text{ if } E=\R_+,\\
	 \a^*&,\text{ if } E=\R,\, i=1,\\
	 \b^*&,\text{ if }E=\R,\,i=2.
        \end{cases}
\end{align*}
Then in any of the three cases $E=[a,b]$, $E=\R_+$, $E=\R$, for any 
$k\geq1$ as $n\to\infty$
            \begin{align*}
                \sqrt{n}\big((m_1^{(n)}, \ldots, m_k^{(n)}) - (m_1^*, \ldots, m_k^*)\big) \xrightarrow{d} \mathcal{N}(0, \Sigma_k),
            \end{align*}
where the matrix  $\Sigma_k$ is given by
            \begin{align*}
                 \Sigma_k = (D\varphi_k^{E}(\vec{y}^*))^t \operatorname{diag}(W_1''(y_1^*), W_2''(y_2^*), W_1''(y_1^*), 
\ldots)^{-1}(D\varphi_k^{E}(\vec{y}^*)).
            \end{align*} 
Here, the maps $\varphi_k^{E}$ have been defined in \eqref{phimap}, \eqref{phimaphalf} and \eqref{phimapreal}, \eqref{phimapreal2}, 
the diagonal matrix is of size $k\times k$  and $\vec{y}^* = (y_1^*, y_2^*, y_1^*, \ldots) \in \mathbb{R}^k$.

In the case $E=\R_+$ and  $z_1^* = z_2^*$, we have
            \begin{align*}
                (D\varphi_k^{\mathbb{R}_+}(\vec{y}^*))_{i, j} = (z_1^*)^{i-1} \left(\binom{2i}{i - j} - \binom{2i}{i - j - 
1}\right).
            \end{align*}
\end{theorem}
Theorem \ref{thrm_CLT} shows that in all considered cases the $1/\sqrt n$-fluctuations of $m_1^{(n)}, \ldots, m_k^{(n)}$ around $m_1^*, \ldots, m_k^*$ are Gaussian. We will now study larger fluctuations. The appropriate tool for describing the exponentially small probabilities associated to these fluctuations is the large deviations principle. Recall that a sequence of random vectors $(X_n)_n$ with values in a Polish space $\mathcal X$ is said to satisfy a large deviations principle with speed $(b_n)_n, \lim_{n\to\infty}b_n=\infty$, and good rate function $I$, if $I:\mathcal X\to[0,\infty]$ is lower semi-continuous, has compact level sets \mbox{$\{x\in\mathcal X:I(x)\leq K\}, K\geq0$} and for any open set $O\subset\mathcal X$ and closed set $U\subset\mathcal X$
\begin{align}
     \liminf_{n\to\infty}\frac{1}{b_n}\log P(X_n\in O)&\geq-\inf_{x\in O}I(x),\label{LDP_open}\\
		 \limsup_{n\to\infty}\frac{1}{b_n}\log P(X_n\in U)&\leq-\inf_{x\in U}I(x),\label{LDP_closed}
\end{align}
cf.~\cite[p.~6]{DZ}. The next theorem is a result on moderate deviations. It shows that on scales up to $o(1)$ the exponential leading order asymptotics are still given by the Gaussian distributions from Theorem \ref{thrm_CLT}, in particular they are universal.
\begin{theorem}\label{thrm_MDP}
	Let the conditions of Theorem \ref{thrm_CLT} be satisfied. Then for any of the three cases $E=[a,b],\, 
E=\R_+,\, E=\R$, for any real-valued sequence $(a_n)_n$ with $\lim_{n\to\infty}a_n=\infty$ and $a_n=o(\sqrt{n})$, the sequence of 
random variables
	\begin{align*}
	a_n\big((m_1^{(n)}, \ldots, m_k^{(n)}) - (m_1^*, \ldots, m_k^*)\big)
	\end{align*}
	satisfies a large deviations principle on $\R^k$ with speed $b_n = \frac{n}{a_n^2}$ and good rate function
		\begin{align*}
		I(x) := \frac{1}{2}\|\operatorname{diag}(W_1''(y_1^*), W_2''(y_2^*), W_1''(y_1^*), \ldots)^{1/2} 
D\varphi_k^{E}(\vec{y}^*)^{-1} x\|_2^2.
		\end{align*}
\end{theorem}

The next result shows that for fluctuations of order 1 a new, non-universal rate function arises.

\begin{theorem}\label{thrm_LDP}
	Let the conditions of Theorem \ref{thrm_compact_1}, Theorem~\ref{thrm_R_+} or Theorem \ref{thrm_R} be satisfied. Then in 
each of the three cases, the sequence $(m_1^{(n)}, \ldots, m_k^{(n)})_n$ satisfies a large deviations principle on $\mathcal 
M_k(E)$ with speed $n$ and good rate function $I(m):=\infty$ for $m\in\partial \mathcal M_k(E)$ and for 
$m\in\textup{Int}\lb\mathcal M_k(E)\rb$
        \begin{align*}
            I(m) :={}& \sum \limits_{j = 1}^{\lfloor\frac{k + 1}{2}\rfloor}  \big\{ W_1(y_{2j - 1}) -  W_1(y_1^*) \big\} + \sum 
\limits_{j = 1}^{\lfloor\frac{k}{2}\rfloor}  \big\{ W_2(y_{2j})  -  W_2(y_2^*) \big\}.
        \end{align*}    
Here $y_i^*, i=1,2$ are as in Theorem \ref{thrm_CLT} and $y_j,j=1,\dots,k$ are defined similarly as $p_j$ ($E=[a,b]$), $z_j$ 
($E=\R_+$) or for $E=\R$ as $\a_{\frac{j+1}2}$ ($j$ odd) and $\b_{j/2}$ ($j$ even). 
       
\end{theorem}

We remark in passing that the case $E=[0,1]$, $V_1=V_2 \equiv 0$ is Theorem 2.6 in \cite{gamloz2004}.

\section{Proofs}\label{sec4}
\begin{proof}[Proof of Lemma~\ref{phi_real_even}]
 For each vector of moments $(m_1, \ldots, m_{2n})$ in the interior of the moment space $\mathcal{M}_{2n}(\mathbb{R})$, we can 
find a probability measure $\mu$ with infinite support and the first $2n$ moments given by $m_1,\dots,m_{2n}$. It is easy to see 
that the following relationship 
holds between the monic orthogonal polynomials $P_k$ corresponding to $\mu$ and their recursion coefficients $\alpha_i, \beta_i$,
    \begin{align}
        \int x^k P_k(x) \, d\mu(x) &= \beta_1 \cdots \beta_k \label{recursion1}\\
        \int x^{k + 1} P_k(x) \, d\mu(x) &= \beta_1 \cdots \beta_k (\alpha_1 + \cdots + \alpha_{k + 1}).\label{recursion2}
    \end{align}
    From this we can immediately see that $\beta_1, \ldots, \beta_k$ only depend on the moments $m_1, \ldots, m_{2k}$, while 
$\alpha_1, \ldots, \alpha_k$ only depend on the moments $m_1, \ldots, m_{2k - 1}$. On the other hand, we may determine each moment 
$m_{2k}$ from $\beta_1, \ldots, \beta_k, \alpha_1, \ldots, \alpha_k$ and each moment $m_{2k - 1}$ from 
\newline$\beta_1, \ldots, \beta_{k - 1}, \alpha_1, \ldots, \alpha_{k}$. Therefore the mapping $\varphi_{2n}^\mathbb{R}$ in 
\eqref{phimapreal}  is a well-defined bijection between $(\alpha_1, \beta_1, \ldots, \alpha_n, \beta_n)$ and $(m_1, \ldots, 
m_{2n})$. The corresponding Jacobian matrix $D\varphi_{2n}^\mathbb{R} $ is a lower triagonal matrix with  determinant  
given by
    \begin{align*}
        \det D\varphi_{2n}^\mathbb{R} = \prod \limits_{k = 1}^n \left(\frac{\partial m_{2k - 1}}{\alpha_k} \cdot \frac{\partial m_{2k}}{\beta_k}\right).
    \end{align*}
    In order to calculate these derivatives, note that since the $P_k$ are monic orthogonal polynomials we have
    \begin{align*}
        \int x^k P_{k - 1}(x) \, d\mu(x) = m_{2k - 1} + \sum \limits_{i = 0}^{2k - 2} \lambda_i m_i
    \end{align*}
    for some real numbers $\lambda_i$ (that may depend on $k$). Since $m_1, \ldots, m_{2k - 2}$ only depend on $\beta_1, \ldots, 
\beta_{k - 1}, \alpha_1, \ldots, \alpha_{k - 1}$, we get with \eqref{recursion2}
    \begin{align*}
        \frac{\partial m_{2k - 1}}{\partial \alpha_k} = \frac{\partial \int x^k P_{k - 1}(x) \, d\mu(x)}{\partial \alpha_k} = \beta_1\cdots \beta_{k - 1}.
    \end{align*}
    A similar argument using \eqref{recursion1} shows
    \begin{align*}
        \frac{\partial m_{2k}}{\partial \beta_k} = \beta_1 \cdots \beta_{k - 1},
    \end{align*}
    which leads to
    \begin{align*}
        \det D\varphi_{2n}^\mathbb{R} = \prod \limits_{k = 1}^n \prod \limits_{j = 1}^{k - 1} \beta_j^2 = \prod \limits_{j = 1}^{n - 1} \prod \limits_{k = j + 1}^n \beta_j^2 = \prod \limits_{j = 1}^{n - 1} \beta_j^{2n - 2j}.
    \end{align*}
\end{proof}

We will now prove the large deviations principles, as they play an important role in the proofs of Theorems \ref{thrm_compact_1}, \ref{thrm_R_+} and \ref{thrm_R}.
\begin{proof}[Proof of Theorem \ref{thrm_LDP}]
	For the sake of brevity we restrict ourselves to the case  $E=[a,b]$, the remaining cases can be proved analogously. To this extent, we will show 
that each $p_{2i - 1}^{(n)}$ satisfies a large deviations principle on $[0,1]$ with good rate function
	\begin{align}
	I_1(p) := W_1(p) - W_1(p_1^*),\, p\in(0,1),\quad I_1(p):=\infty,\, p\in\{0,1\},\label{I_1}
	\end{align}
	where $W_1(p) = V_1(p) - \log(p(1-p))$. Analogously, the $p_{2i}^{(n)}$ satisfy a large deviations principle on $[0,1]$ 
with good rate function $I_2(p) := W_2(p) - W_2(p_2^*)$ on the interval $(0,1)$ and 
$\infty$ 
elsewhere. The assertion then follows from the independence of the $p_i$'s and the contraction principle. Note that 
$\varphi_k^{[a,b]}$ is bijective and thus the rate function does not change when passing from canonical to ordinary moments.
	
	For the upper bound \eqref{LDP_closed}, let $U \subset [0, 1]$ be a closed set. If $U\subset \{0,1\}$, \eqref{LDP_closed} 
is trivially true by definition of $\mathbb{P}_{n,[a,b],V_{1,2}}$ and thus we may assume $U\cap(0,1)\not=\emptyset$. Then, setting $W^U:= 
\inf \limits_{x \in U} W_1 
(x) $,
	\begin{align*}
	&\limsup \limits_{n \to \infty}\frac{1}{n} \log \int_U e^{-nV_1(x)+(n-i)\log(x(1-x))}\, dx \le{} \limsup \limits_{n \to \infty} \frac{1}{n} \log \int_0^1 e^{-iV_1(x)-(n-i)W^U}\, dx  = -W^U.
	\end{align*}
	For the lower bound \eqref{LDP_open}, let $O \subset [0, 1]$ be an open set and define $W^O := \inf \limits_{x \in O} W_1(x)$. Let $\varepsilon > 0$ be arbitrary. By continuity of $W$ on the interval $(0,1)$ and openness of $O$ we know that $O \cap \{W_1 < W^O + \varepsilon\}$ is a nonempty open set. This yields
	\begin{align*}
	&\liminf \limits_{n \to \infty} \frac{1}{n} \log \int_O e^{-nV_1(x)+(n-i)\log(x(1-x))}  \, dx \\
	\ge{}& \liminf \limits_{n \to \infty} \frac{1}{n} \log \int\limits_{O \cap \{W_1 < W^O + \varepsilon\}} e^{-nV_1(x)+(n-i)\log(x(1-x))} \, dx \\
	\ge{}& \liminf \limits_{n \to \infty} \frac{1}{n} \log \int \limits_{O \cap \{W_1 < W^O + \varepsilon\}} e^{-iV_1(x)-(n-i)(W^O + \varepsilon)} \, dx = -W^O - \varepsilon.
	\end{align*}
	Now let $\varepsilon \to 0$, then the assertion finally follows from the choice $U=O = [0,1]$ which shows that the normalization constant of the density satisfies
	\begin{align*}
	\lim \limits_{n \to \infty} \frac{1}{n} \log \int_0^1 e^{-nV_1(x)+(n-i)\log(x(1-x))}\, dx = -\inf \limits_{y \in (0, 1)} W_1(y).
	\end{align*}
\end{proof}

Next, we will prove the results on laws of large numbers in Section \ref{sec2}. It follows from Theorem \ref{thrm_LDP} and the 
Borel-Cantelli lemma that in all three cases $(m_1^{(n)},\dots,m_k^{(n)})\to(m^*_1,\dots,m^*_k)$ almost surely as $n\to\infty$, 
where 
$m^*_j$ are determined by $p^*_i,z^*_i, i=1,2$ or $\a^*,\b^*$, respectively. The convergence in $L^1$ follows  for $E=[a,b]$ 
immediately by the boundedness of the moments. For unbounded $E$, it suffices to see that the $m_j^{(n)}$'s are uniformly integrable 
thanks to the exponential decay from the large deviations principle.  It remains to identify the corresponding measures to the 
moment sequences $(m_1^*,m_2^*,\dots)$. The general technique to do this is to consider the Jacobi operator associated to the 
recurrence coefficients of the orthogonal polynomials and derive an equation for the Stieltjes transform of the desired 
measure via a continued 
fraction expansion. We start with the simplest case of Theorem \ref{thrm_R}, where we explain the strategy in detail.

We will make use of the following lemma.
\begin{lemma}\label{lemma_Stieltjes_expansion}
 Let $\mu$ be a Borel probability measure on $\R$ that is determined by its moments (i.e.~the Hamburger moment problem to the moments of $\mu$ is determinate). Let $\a_1,\b_1,\a_2,\b_2\dots$ denote the 
recurrence 
coefficients of the monic orthogonal polynomials to the measure $\mu$ (see \eqref{recurrence}). If $\mu$ is supported on $N$ 
points, we set $\b_j:=0$ for $j\geq N$. Then the Stieltjes 
transform of $\mu$, 
\begin{align*}
\Phi(z):=\int\frac{d\mu(x)}{z-x},
\end{align*} 
 defined for $z\in\C^+:=\{z\in\C:\Im z>0\}$, has the continued fraction expansion
\begin{align*}
\Phi(z)=\ &\polter{1}{z-\a_1}-\polter{\b_1}{z-\a_2}-\polter{\b_2}{z-\a_3}-\dots.
\end{align*}
Here the convergents 
$$\polter{1}{z-\a_1}-\polter{\b_1}{z-\a_2}-\dots-\polter{\b_l}{z-\a_{l+1}}$$
 converge 
locally uniformly in $\C^+$ as $l\to\infty$. 
\end{lemma}
Although the connection between continued fractions, Stieltjes transforms and orthogonal polynomials is classical and this result should be well-known, we did not 
manage to find this lemma in the literature. For measures with compact support, it is called Markov's theorem. We will give an 
elementary derivation.
\begin{proof}[Proof of Lemma \ref{lemma_Stieltjes_expansion}]

 Let $\mu$ be a measure whose support consists of precisely $N$ distinct points. Then the  monic orthogonal 
polynomials $P_1, \ldots, P_{N}$  up to order $N$ with respect to $\mu$ and the corresponding 
recursion coefficients $\alpha_1,\b_1,\a_2,\b_2,\ldots,\b_{N-1},\a_N$ are well-defined. 
Moreover, if $\mu$ has masses $\omega_1, \ldots , \omega_N$ at the 
points $t_1,\ldots , t_N$ and $m_j$ denotes the $j$-th moment of $\mu$, the monic orthogonal polynomial $P_{N}$ is proportional to the polynomial 
\begin{align*}
\tilde P_{N} (t) &= \det \left(  \begin{matrix}  
1 & m_1 & \ldots & m_{N-1} &1 \\
 m_1 & m_2 & \ldots & m_{N} & t \\
\vdots  & \vdots   & \ddots & \vdots & \vdots  \\
 m_N & m_{N+1} & \ldots & m_{2N-1} & t^N 
 \end{matrix}
 \right) \\
 &=
 \sum_{i_0=1}^N \ldots  \sum_{i_{N-1} = 1}^N \omega_{i_0}  \ldots \omega_{i_{N-1}} t_{i_1}^1 t_{i_2}^2\ldots  t_{i_{N-1}}^{N-1} 
 \det \left(  \begin{matrix}  
1 & 1 & \ldots & 1 &1 \\
t_{i_0}  & t_{i_1} & \ldots & t_{i_{N-1}} & t \\
\vdots  & \vdots   & \ddots & \vdots & \vdots  \\
 t_{i_0}^N  & t_{i_1} ^N& \ldots & t_{i_{N-1}}^N & t^N 
 \end{matrix}
 \right)~.
 \end{align*} 
 Now the determinant in the last line vanishes whenever two indices $i_j$ and $i_k$ coincide. If all indices are different, the determinant is equal (up to a sign) to the polynomial 
 $\ell (t) = \prod_{i=1}^N (t-t_i)$. 
 Consequently, the polynomials $\tilde P_{N} $ and $P_{N} $ are also proportional 
 to  $\ell (t) $ and  therefore vanish precisely at the the support points $t_1,\ldots t_N$ of the measure $\mu$.
 
We now define for $z \in \mathbb{C}^+$ the continued fraction
    \begin{align*}
        f_j(z) := \polter{1}{z - \alpha_1} - \polter{\beta_1}{z - \alpha_2} - \polter{\beta_2}{z - \alpha_3} - \dots - 
\polter{\beta_{j - 1}}{z - \alpha_j}, \quad j=1,\dots,N.
    \end{align*}
   Writing $f_j(z)$ as a single fraction $\frac{A_{j}(z)}{B_{j}(z)}$, we see that $A_j(z)$ and $B_j(z), j=1,\dots,m$ satisfy the recursions 
 $A_0(z) 
:= 0, B_0(z) :=1$, $ A_1(z) := 1, B_1(z) := z - \alpha_1$ and 
    \begin{align*}
        A_j(z) &= (z - \alpha_j) A_{j - 1}(z) - \beta_{j - 1} A_{j - 2}(z), \\
        B_j(z) &= (z - \alpha_j) B_{j - 1}(z) - \beta_{j - 1} B_{j - 2}(z)
    \end{align*}
    for $2 \le j \le N$. 
    Clearly, $B_j$ is a polynomial in $z$ of degree $j$ with leading coefficient $1$ 
    and as it satisfies the same recursion as the 
orthogonal polynomials $P_j$, we conclude $B_j = P_j$ for $0 \le j \le N$. Furthermore, note that the sequence of functions
    \begin{align*}
        Q_j(z) := \int \frac{P_{j}(z) - P_{j}(t)}{z-t} \, d\mu(t)
    \end{align*}
    satisfies the same recursion as $A_j$, from which we can conclude $Q_j = A_j$ for $0 \le j \le 
N$. As the  roots of $P_{N}$ are precisely 
the support points of the measure $\mu$ we obtain 
    \begin{align*}
        f_N(z) = \frac{A_N(z)}{B_N(z)} = \frac{1}{P_N(z)} \int \frac{P_N(z)}{z - t} \, d\mu(t) = \int \frac{1}{z - t} \, d\mu(t), 
    \end{align*}
  which  concludes the proof for a measure $\mu$ with  finite support. 
		
		If $\mu$ has infinite support, all recursion coefficients 
$\beta_j$ are strictly positive. Let $N$ be an arbitrary natural number.  There is a unique measure $\mu_N$ supported on $N$ points such that
the corresponding monic
orthogonal polynomials have the recursion 
coefficients $\alpha_1, \beta_1,\dots,\b_{N-1},\a_N$. By the arguments above, the Stieltjes transform of $\mu_N$ has the form
    \begin{align*}
        f_N(z) = \polter{1}{z - \alpha_1} - \polter{\beta_1}{z - \alpha_2} - \polter{\beta_2}{z - \alpha_3} - \dots - 
\polter{\beta_{N - 1}}{z - \alpha_N}.
    \end{align*}
    Since the recursion coefficients up to order $N$ determine the moments of $\mu_N$ up to order $2N-1$,  we know that 
$m_j(\mu_N) = m_j(\mu)$ for $1 \le j \le 2N - 1$. Letting $N\to\infty$ thus shows $\lim_{N\to\infty} m_j(\mu_N) =m_j(\mu)$ 
for all $j$. Since the measure $\mu$ is uniquely determined by its moments, this implies the weak convergence $\mu_N \xrightarrow{w} 
\mu$. For any fixed $z \in \mathbb{C}^+$, the function $t \mapsto \frac{1}{z - t}$ is a bounded continuous function. Therefore the Stieltjes transform of $\mu_N$ converges to the Stieltjes transform of $\mu$, i.e.
    \begin{align*}
        \int \frac{1}{z - t} \, d\mu(t) = \lim \limits_{N \to \infty} \int \frac{1}{z - t} \, d\mu_N(t) = \polter{1}{z - \alpha_1} - 
\polter{\beta_1}{z - \alpha_2} - \polter{\beta_2}{z - \alpha_3} - \dots.
    \end{align*}
As $z\mapsto\frac1{z-t}$ is analytic in $\C^+$ and uniformly bounded away from the real line, $f_N$ is analytic in $\C^+$ and for any compact $K\subset \C^+$ we have $\sup_{N,z\in K}\lv f_N(z)\rv\leq M$ for some $M>0$. It follows by Montel's theorem that the convergence is uniform on $K$.

\end{proof}

\begin{proof}[Proof of Theorem \ref{thrm_R}]

	Let $\mu_{SC,\a^*,\b^*}$ be the measure for which the recurrence coefficients of the associated monic orthogonal polynomials are $\a_j=\a^*$ and $\b_j=\b^*$ for all $j$. From \eqref{phimapreal} we know that $\mu_{SC,\a^*,\b^*}$ has finite moments. By Carleman's criterion (in terms of recurrence coefficients, see \cite[p. 59]{ShohatTamarkin}, the Hamburger moment problem for the moments of $\mu_{SC,\a^*,\b^*}$ is determinate, if
\begin{align}
\sum_{j=1}^\infty \frac{1}{\sqrt\b_j}=\infty,\label{Carleman}
\end{align}
	which is clearly the case here.
	Thus by Lemma \ref{lemma_Stieltjes_expansion} the Stieltjes transform
	\begin{align*}
	\Phi_{SC,\a^*,\b^*}(z):=\int\frac{d\mu_{SC,\a^*,\b^*}(x)}{z-x},
	\end{align*}
	 has the continued fraction expansion
	\begin{align}
	\Phi_{SC,\a^*,\b^*}(z)=\ &\polter{1}{z-\a^*}-\polter{\b^*}{z-\a^*}-\dots=\frac{1}{z-\a^*-\b^*\Phi_{SC,\a^*,\b^*}(z)},\label{cont_fract_sc}
	\end{align}
	where the dots $\dots$ in \eqref{cont_fract_sc} mean a continued repetition of the last fraction before the dots. Solving algebraically for $\Phi_{{SC,\a^*,\b^*}}(z)$ yields the two solutions
	\begin{align*}
	\frac{z-\a^*\mp\sqrt{(z-\a^*)^2-4{\b^*}}}{2\b^*}.
	\end{align*}
	Since any Stieltjes transform maps the upper half plane to the lower half plane, we get
	\begin{align}
		\Phi_{SC,\a^*,\b^*}(z)=\frac{z-\a^*-\sqrt{(z-\a^*)^2-4{\b^*}}}{2\b^*},\label{Stieltjes_sc}
	\end{align}
	where we define $\sqrt{(z-\a^*)^2-4{\b^*}}$ for $z\in\C^+$ as the branch with positive imaginary part. Note that 
$\sqrt{(z-\alpha)^2 - 4\beta}$ admits a continuous extension from $\C^+$ to 
$\mathbb{R}$ via
    \begin{align*}
        \lim_{y\to0+}\sqrt{(x+iy-\alpha)^2 - 4\beta} = 
        \begin{cases}
						-\sqrt{(x-\alpha)^2 - 4\beta} &, x < \alpha - 2\sqrt{\beta}\\
            i \sqrt{4\beta - (x-\alpha)^2} &, x\in [\alpha - 2\sqrt{\beta}, \alpha + 2\sqrt{\beta}] \\
            \sqrt{(x-\alpha)^2 - 4\beta} &, x >\alpha + 2\sqrt{\beta}
        \end{cases}.
    \end{align*}
    Thus $\Phi_{SC,\a^*,\b^*}$ has a continuous extension from the upper half 
plane to the real line and $\mu_{\a^*,\b^*}$ has a density on $\R$ which is given by the Stieltjes inversion formula (see 
e.g.~\cite[Remark 2.20]{nicspe2006}) 
	\begin{align}
	\frac{\mu_{\a^*,\b^*}(dx)}{dx}&=-\frac1\pi\lim_{y\to0+}\Im\Phi_{SC,\a^*,\b^*}(x+iy)\label{Stieltjes_inversion}\\
	&=\frac{1}{2\pi\b^*}\sqrt{4\b^*-(x-\a^*)^2}1_{[\a^*-2\sqrt{\b^*},\a^*+2\sqrt{\b^*}]}(x).\notag
	\end{align}
	It is well-known that (see \cite[Corollary 2.14]{nicspe2006}) the $j$-th moment of the semicircle distribution 
$\mu_{SC,0,1}$ is $\frac{1}{j+1}\binom{2j}{j}$, \eqref{moments_SC} follows by a simple computation.
\end{proof}

\begin{proof}[Proof of Theorem \ref{thrm_compact_1}]
	Let $\mu_{p^*_1,p^*_2}$ be the probability measure determined by having canonical odd moments $p^*_1$ and canonical even 
moments $p^*_2$. For a probability measure on $[a,b]$ with canonical moments $p_1,p_2, p_3, \ldots$ the recurrence coefficients of 
its monic orthogonal polynomials are given by (cf. \cite[Corollary 2.3.4, eq.~(1.4.6)]{dettstud1997})
	\begin{align*}
	&\a_j=a+(b-a)(q_{2j-3}p_{2j-2}+q_{2j-2}p_{2j-1}),\\
	&\b_j=(b-a)^2q_{2j-2}p_{2j-1}q_{2j-1}p_{2j}, j=1,\dots.
	\end{align*}
	Here we set $p_{-1}=p_0=0$ and as usual $q_j:=1-p_j$. In our case  $\a_1=a+(b-a)p^*_1$, $\b_1=(b-a)^2p^*_1q^*_1p^*_2$,  and for $j\geq 2$ we have $\a_j=a+(b-a)(p_1^*q_2^*+p_2^*q_1^*)$, $\b_j=(b-a)^2p^*_1q^*_1p^*_2q^*_2$. Since $[a,b]$ is compact, the moment problem is determinate and hence Lemma \ref{lemma_Stieltjes_expansion} yields that the Stieltjes transform 
	\begin{align*}
	\Phi_{p^*_1,p^*_2}(z):=\int\frac{d\mu_{p^*_1,p^*_2}(x)}{z-x}
	\end{align*}
	 has the continued fraction expansion
	\begin{align*}
	\Phi_{p^*_1,p^*_2}(z)={}&\ \polter{1}{z-a-(b-a)p_1^*}-\polter{(b-a)^2{p_1^*}q^*_1p_2^*}{z-a-(b-a)(p_1^*q_2^*+p_2^*q_1^*)}\\
	&-\polter{(b-a)^2p_1^*q_1^*p_2^*q_2^*}{z-a-(b-a)(p_1^*q_2^*+p_2^*q_1^*)}-\dots,\\
	={}&\polter{1}{z-a-(b-a)p_1^*}-(b-a)^2{p_1^*}q^*_1p_2^*\Phi_{SC,\a,\b}(z),
	\end{align*}
	where $\Phi_{SC,\a,\b}$ is from \eqref{cont_fract_sc} with $\a:=a+(b-a)(p_1^*q_2^*+p_2^*q_1^*)$, $\b:=(b-a)^2p^*_1q^*_1p^*_2q^*_2$. Thus by \eqref{Stieltjes_sc}
	\begin{align*}
	\Phi_{p^*_1,p^*_2}(z)&=\frac{2q_2^*}{2q_2^*(z-a-(b-a)p_1^*)-(z-\a-\sqrt{(z-\a)^2-4\b})}\\
	&=\frac{(1-2p_2^*)z+\a-2{q_2^*}(a+(b-a)p_1^*)-\sqrt{(z-\a)^2-4\b}}{2p_2^*(z-a)(b-z)}.
	\end{align*}
	As atoms of $\mu_{p^*_1,p^*_2}$ are simple poles of the Stieltjes transform, atoms  can only 
be at $a$ or $b$.  They can be identified using the formula 
\begin{align}
\mu_{p^*_1,p^*_2}(\{x\})=-\lim_{y\to0+}y\Im\Phi_{p^*_1,p^*_2}(x+iy).\label{Stieltjes_atom}
\end{align}
Using this, we get after some algebra for $x=a$
\begin{align*}
\mu_{p^*_1,p^*_2}(\{a\})=\frac{p_2^*-p_1^*+\lvert p_2^*-p_1^*\rvert}{2p_2^*}=\begin{cases}
                                                                                    0,&\quad\text{if }p_1^*\geq p_2^*\\
										    1-\frac{p_1^*}{p_2^*},&\quad \text{if 
}p_1^*<p_2^*
                                                                                   \end{cases}
.
\end{align*}
For $x=b$, we get similarly
\begin{align*}
\mu_{p^*_1,p^*_2}(\{b\})=\frac{p_1^*+p_2^*-1+\lvert 1-p_1^*-p_2^*\rvert}{2p_2^*}=\begin{cases}
                                                                                    0,&\quad\text{if }p_1^*+p_2^*\leq1,\\
										    \frac{p_1^*+p_2^*-1}{p_2^*},&\quad \text{if 
}p_1^*+p_2^*>1.
                                                                                   \end{cases}
\end{align*}

$\Phi_{p^*_1,p^*_2}(z)$ has a continuous extension to $\R\setminus\{a,b\}$. Thus the measure is absolutely continuous on 
$\R\setminus\{a,b\}$ and the density of the absolutely continuous part $\mu_{p_1^*,p_2^*}^{ac}$  can 
be computed using \eqref{Stieltjes_inversion} as
\begin{align*}
\frac{\mu_{p^*_1,p^*_2}^{ac}(dx)}{dx}=\frac{\sqrt{4\b-(\a-x)^2}}{2\pi p_2^*(x-a)(b-x)}
\end{align*}
for $x\in[\a-2\sqrt{\b},\a+2\sqrt\b]$, and $0$ elsewhere. This proves (1), since $l_\pm=\a\pm2\sqrt\b$.

For (2) we use the well-known fact from potential theory (cf.~e.g.~\cite[Theorem I.3.3]{SaffTotik})  that $\mu$ is the minimizing 
measure of \eqref{eq_measure} if and only if it satisfies the Euler-Lagrange equations
\begin{align}\label{Euler-Lagrange}
Q(t)-2\int \log\lv t-s\rv d\mu(s)\begin{cases}
=l&,\quad\text{if }t\in\textup{supp}(\mu),\\
\geq l&,\quad\text{if }t\notin\textup{supp}(\mu),
\end{cases}
\end{align}
where $l$ is a real constant. Differentiating, we get for $t\in\textup{supp}(\mu)$
\begin{align}\label{Q'}
Q'(t)=2H_\mu(t),
\end{align}
where 
\begin{align*}
H_{\mu}(t):=\int\frac{d\mu(s)}{t-s}
\end{align*}
is the Hilbert transform of $\mu$. Note that the integral is understood as a principal value integral. 
The Hilbert transform of an absolutely continuous measure can be obtained from its Stieltjes transform $\Phi_\mu$ via (see 
e.g.~\cite[p.~94]{hiaipetz2000})
\begin{align*}
H_{\mu}(t)&=\lim_{y\to0+}\Re\Phi_{\mu}(t+ iy).
\end{align*}
In our case this gives together with \eqref{Q'}
\begin{align*}
Q'(t)&=\frac{(1-2p_2^*)t+\a-2{q_2^*}(a+(b-a)p_1^*)}{p_2^*(t-a)(b-t)}=-\frac{p_1^*-p_2^*}{p_2^*(t-a)}+\frac{
1-p_1^*-p_2^*}{p_2^*(b-t)}.
\end{align*}
Integration yields
\begin{align}
Q(t)={}&-\lb\frac{p_1^*}{p_2^*}-1\rb\log(t-a)-\lb\frac{1-p_1^*-p_2^*}{p_2^*}\rb\log(b-t)\label{field}
\end{align}
on the support. The integration 
constant does not matter here and thus is set to $0$ for simplicity. We will consider $Q$ defined by \eqref{field} as function 
$Q:[a,b]\rightarrow \R\cup\{+\infty\}$. By construction, $Q$ satisfies the equation of 
\eqref{Euler-Lagrange} on the support of $\mu_{p^*_1,p^*_2}$. For the inequality in \eqref{Euler-Lagrange}, we compute the Hilbert 
transform $H_{\mu_{p_1^*,p_2^*}}$ outside of the support of $\mu_{p_1^*,p_2^*}$ as
\begin{align*}
H_{\mu_{p_1^*,p_2^*}}(t)=
    \begin{cases}
        \frac{Q'(t)}2+\frac{\sqrt{(t-\a)^2-4\b}}{2p_2^*(t-a)(b-t)}\geq \frac{Q'(t)}2 &, t \le \alpha - 2\sqrt{\beta}, \\
        \frac{Q'(t)}2-\frac{\sqrt{(t-\a)^2-4\b}}{2p_2^*(t-a)(b-t)}\leq \frac{Q'(t)}2 &, t \ge \alpha + 2\sqrt{\beta}.
    \end{cases}
\end{align*}
 Consequently, $Q(t) - 2\int \log|t-s| \, d\mu_{p_1^*, p_2^*}(s)$ is nonincreasing on $[a, 
l_-)$, constant on $[l_-, l_+]$ and nondecreasing on $(l_+, b]$. This implies the inequality in \eqref{Euler-Lagrange} and thus 
proves (2).
\end{proof}

\begin{proof}[Proof of Theorem \ref{thrm_R_+}]
 It is not difficult to see that the recurrence coefficients for the orthogonal polynomials to a probability measure $\mu$ on $\R_+$ with canonical moments $z_1,z_2,\dots$ are given by
\begin{align*}
&\a_j=z_{2j-2}+z_{2j-1},\\
	&\b_j=z_{2j-1}z_{2j}, j\geq1
\end{align*}
with the convention $z_0:=0$. 

Let $\mu_{MP,z_1^*,z_2^*}$ be the probability measure on $\R_+$ with canonical moments $z_{2j-1}=z_1^*$ and $z_{2j}=z_2^*$, $j=1,\dots$. Then the recursion coefficients of the corresponding orthogonal polynomials are $\a_1=z_1^*,\a_j=z_1^*+z_2^*,j\geq2$ and $\b_j=z_1^*z_2^*,j\geq1$. The Stieltjes transform of $\mu_{MP,z_1^*,z_2^*}$ will be denoted by $\Phi_{MP,z_1^*,z_2^*}$. By \eqref{Carleman}, the moment problem is determinate and thus $\Phi_{MP,z_1^*,z_2^*}$ admits the continued fraction expansion
\begin{align*}
\Phi_{MP,z_1^*,z_2^*}(z)&=\ \polter{1}{z-z_1^*}-\polter{z^*_1z_2^*}{z-(z_1^*+z_2^*)} - \quad \ldots  \qquad 
= \quad \frac{1}{z-z_1^*-z^*_1z_2^*\Phi_{SC,\a,\b}(z)},
\end{align*}
where $\Phi_{SC,\a,\b}(z)$ is from \eqref{cont_fract_sc} with $\a:=(z_1^*+z_2^*)$ and $\b=z_1^*z_2^*$. Using \eqref{Stieltjes_sc}, this gives
\begin{align*}
\Phi_{MP,z_1^*,z_2^*}(z)&=\frac{2\b}{2\b(z-z_1^*)-z_1^*z_2^*(z-\a-\sqrt{(z-\a)^2-4\b})}\\
&=\frac{z-z_1^*+z_2^*-\sqrt{(z-\a)^2-4\b}}{2z_2^*z}.
\end{align*}
Clearly, $\mu_{MP,z_1^*,z_2^*}$ can have an atom only at 0. A computation using \eqref{Stieltjes_atom} gives
\begin{align*}
\mu_{MP,z^*_1,z^*_2}(\{0\})=\frac{z_2^*-z_1^*-\lvert z_1^*-z_2^*\rvert}{2z_2^*}=\begin{cases}
                                                                                    0,&\quad\text{if }z_2^*\geq z_1^*,\\
										    1-\frac{z_1^*}{z_2^*},&\quad \text{if 
}z_2^*<z_1^*.
                                                                                   \end{cases}
\end{align*}
The density of the absolutely continuous part can again be determined using \eqref{Stieltjes_inversion} as
\begin{align*}
\frac{\mu_{MP,z^*_1,z^*_2}(dx)}{dx}=\frac{\sqrt{4\b-(\a-x)^2}}{2\pi z_2^*x}
\end{align*}
for $x\in[\a-2\sqrt{\b},\a+2\sqrt\b],\,x\not=0$, and $0$ elsewhere. Now the statement of the theorem follows noting 
$l_\pm=\a\pm2\sqrt\b$.
\end{proof}

\begin{proof}[Proof of Theorem \ref{thrm_CLT}] We will only prove the case $E=\R_+$, as the remaining parts are shown by  similar 
arguments.
Consider a moment vector under the distribution  $\mathbb{P}_{n,\R_+,V_{1,2}}$ defined by the density \eqref{def_R_+}.
    We will show that the  canonical moments satisfy
    \begin{align*}
        \sqrt{n}(z_{2i - 1}^{(n)} - z_1^*) &\xrightarrow{d} \mathcal{N}(0, W_1''(z_1^*)^{-1}) \\
        \sqrt{n}(z_{2i}^{(n)} - z_2^*) &\xrightarrow{d} \mathcal{N}(0, W_2''(z_2^*)^{-1})
    \end{align*}
    as $n \to \infty$. 
    The assertion of the theorem then follows from the independence of the $z_i^{(n)}$ and an application of the delta-method.
    
    By Scheff\'{e}'s Lemma, weak convergence of a sequence of measures can be proved by showing pointwise convergence of the corresponding densities. The density of $\sqrt{n}(z_{2i - 1}^{(n)} - z_1^*)$ is given by
    \begin{align*}
    f_n(x)  :=  \frac{g_n(x)}{c_n}, 
    \end{align*}
    where 
      \begin{align*}
     g_n(x)  :=  \exp\big\{{-}n(W_1(z_1^* + \tfrac{x}{\sqrt{n}}) - W_1(z_1^*))\big\} (z_1^* + \tfrac{x}{\sqrt{n}})^{-(2i - 1)} 1_{\left\{z_1^* + \tfrac{x}{\sqrt{n}} > 0\right\}}
    \end{align*}
and    $c_n$ is an appropriate normalization constant.
       By Taylor's theorem we obtain that
    \begin{align*}
        W_1(z_1^* + x/\sqrt{n}) = W_1(z_1^*) + \frac{x^2}{2n} W_1''(z_1^* + \lambda x/\sqrt{n})
    \end{align*}
    holds for some $\lambda \in [0, 1]$. From this we can easily conclude
    \begin{align*}
        g_n(x)  \xrightarrow{n \to \infty} \exp(-W_1''(z_1^*) x^2/2) (z_1^*)^{-(2i - 1)},
    \end{align*}
    and it remains to prove the convergence of the normalization constant. By assumption \newline$W_1''(z_1^*)\not=0$ and since $z_1^*$ is a minimizer of $W_1$, we have $W_1''(z_1^*)>0$. Hence we may choose $0 < \varepsilon < z_1^*$ so small that the inequality $W_1''(x) > W_1''(z_1^*)/2$ is satisfied for all $x$ with $|x - z_1^*| < \varepsilon$. This yields
    \begin{align*}
        c_n &= \int\limits_{-z_1^*\sqrt{n}}^{\infty} \exp\big\{{-}n(W_1(z_1^* + x/\sqrt{n}) - W_1(z_1^*))\big\} (z_1^* + x/\sqrt{n})^{-(2i - 1)} \, dx \\
        &= \int \limits_{-\varepsilon\sqrt{n}}^{\varepsilon\sqrt{n}} \exp \big\{{-}n(W_1(z_1^* + x/\sqrt{n}) - W_1(z_1^*))\big\} (z_1^* + x/\sqrt{n})^{-(2i - 1)}\, dx + o(1) \\
        &\xrightarrow{n \to \infty} \int \exp\big\{{-}W_1''(z_1^*) x^2/2\big\}(z_1^*)^{-(2i - 1)} \, dx = \sqrt{\frac{2\pi}{W_1''(z_1^*)}} (z_1^*)^{-(2i - 1)}.
    \end{align*}
    Here we have used the dominated convergence theorem with dominating function 
    \begin{align*}
        g(x) := \exp\big\{{-}W_1''(z_1^*) x^2/4\big\} (z_1^* - \varepsilon)^{-(2i - 1)}.
    \end{align*}
    The $o(1)$ term stems from the fact that outside of $(-\varepsilon \sqrt{n}, \varepsilon \sqrt{n})$ the function $W_1(z_1^* + x/\sqrt{n}) - W(z_1^*)$ is bounded from below by some positive constant $K > 0$. The remaining integral can then be bounded by
    \begin{align*}
        \sqrt{n} \exp(-(n - (2i - 1))K) \int_0^\infty \exp\big\{-(2i - 1)(V_1(x) - V_1(z_1^*) + \log(z_1^*(1 - z_1^*)))\big\} \, dx = o(1).
    \end{align*}
    Hence the density $f_n$ converges pointwise  to a centered normal distribution with variance $1/ W_1''(z_1^*)$, which completes 
the proof of the 
    first part of the theorem.
    
    It remains to determine the entries of $D\varphi_k^{\mathbb{R}_+}$ in the case $z_1^* = z_2^*$. In order to do this, we will follow the arguments in \cite{detnag2012}. Therein, a double sequence $g_{i, j}$ is defined by
    \begin{align*}
        g_{i, j} := 
        \begin{cases}
            1 &, \text{ if } i = 0, \\
            0 &, \text{ if } i \ne 0, i > j, \\
            g_{i, j - 1} + z_{j - i + 1} g_{i - 1, j} &, \text{ if } i \ne 0, i \le j.
        \end{cases}
    \end{align*}
    An induction argument over the sum $i + j$ shows that $g_{i, j}$ is a homogeneous polynomial of degree $i$ in $z_1, z_2, 
\ldots$. Consequently, the partial derivative $\frac{dg_{i, j}}{d z_k}$ is a homogeneous polynomial of degree $i - 1$.
    Following the arguments of \cite{detnag2012} we have $g_{k, k} = m_k$ with
    \begin{align*}
        \frac{d m_i}{d z_r}(1, 1, \ldots) = \binom{2i}{i - r} - \binom{2i}{i - r - 1}
    \end{align*}
    and thus	
    \begin{align*}
        \frac{d m_i}{d z_r}(z_1^*, z_1^*, \ldots) = (z_1^*)^{i - 1} \frac{d m_i}{d z_r}(1, 1, \ldots) = (z_1^*)^{i - 1} 
\left(\binom{2i}{i - r} - \binom{2i}{i - r - 1}\right).
    \end{align*}
\end{proof}

\begin{proof}[Proof of Theorem \ref{thrm_MDP}]
	We will only prove the case $E=[a,b]$, the remaining cases are treated  similarly.
	We will first show that each $a_n(p_{2j - 1}^{(n)} - p_1^*)$ satisfies a large deviations principle with good rate function $J(x) := W_1''(p_1^*) x^2/2$ and speed $b_n$, where $(a_n)_n$ and $(b_n)_n$ are chosen as in Theorem \ref{thrm_MDP}. In order to see this, let $U \subset \mathbb{R}$ be an arbitrary closed set and $0 < \varepsilon < 1$ sufficiently small so that $W_1''(y) \ge M > 0$ holds for all $y \in (p^*-\varepsilon, p^*+\varepsilon)$ and some constant $M > 0$. Set $\gamma := \inf \limits_{x \in U} |x|$,  $R(p) := (p(1-p))^{-(2i - 1)}$ and let $I_1$ be the function \eqref{I_1}. Note that $I_1\geq0$ with unique zero $p_1^*$ and $I_1''=W_1''$. For \eqref{LDP_closed} we show first 
	\begin{align*}
        \limsup \limits_{n \to \infty} \frac{1}{b_n} \log \int_U e^{-nI_1(x/a_n + p_1^*)} R(x/a_n + p_1^*) \, dx  \le -W_1''(p^*)\frac{\gamma^2}{2}.
	\end{align*}
	The case $\gamma = \infty$ is trivial, since then $U=\emptyset$, so we may assume $\gamma < \infty$. We will first consider $U\cap \{|x| \geq \varepsilon a_n\}$. We get
	\begin{align*}
        &\limsup \limits_{n \to \infty} \frac{1}{b_n} \log \int_{U} 1_{\{|x| \geq \varepsilon a_n\}}e^{-nI_1(x/a_n + p_1^*)} R(x/a_n + p_1^*) \, dx \\
        \le{}& \limsup \limits_{n \to \infty} \frac{1}{b_n} \log \int_\R 1_{\{|x| \geq \varepsilon a_n\}} e^{-(2i - 1)V_1(x/a_n + p_1^*)} \exp\Big({-}(n-(2i - 1)) \inf \limits_{|y - p_1^*| \geq \varepsilon} I_1(y)\Big) \, dx \\
        \le{}& \limsup \limits_{n \to \infty} \frac{1}{b_n} \log \int_\R a_n e^{-(2i - 1)V_1(t)} \exp\Big({-}(n-(2i - 1)) \inf \limits_{|y - p_1^*| \geq \varepsilon} I_1(y)\Big) \, dt \\
        \le{}& \limsup \limits_{n \to \infty} {a_n^2\Big(\log a_n - (n - (2i - 1)) \inf \limits_{|y - p_1^*| \geq \varepsilon} I_1(y)\Big)}/{n} = -\infty.
	\end{align*}
	For the  set $U\cap \{|x| < \varepsilon a_n\}$, note that by Taylor's theorem
	\begin{align*}
        &\int_U 1_{\{ |x| < \varepsilon a_n\}} e^{-nI_1(x/a_n + p_1^*)} R(x/a_n + p_1^*) \, dx \\
        \le{}&\sup \limits_{|y - p_1^*| < \varepsilon} R(y) \int_U 1_{\{ |x| < \varepsilon a_n\}} \exp\big({-}nx^2/(2a_n^2) \inf \limits_{|z - p_1^*| < \varepsilon}W_1''(z)\big) \, dx \\
        \le{}&\sup \limits_{|y - p_1^*| < \varepsilon} R(y) \int_\R \exp\Big({-}\big((1 - \varepsilon) n\gamma^2/(2a_n^2) + \varepsilon n x^2/(2a_n)\big)\inf \limits_{|z - p_1^*| < \varepsilon}W_1''(z)\Big) \, dx \\
        \le{}& \sup \limits_{|y - p_1^*| < \varepsilon} R(y)\exp\big({-}(1 - \varepsilon) 
        b_n\gamma^2/2 \inf \limits_{|z - p_1^*| < \varepsilon}W_1''(z)\big) \sqrt{{2\pi}\Big / \Big ( \varepsilon{b_n \inf \limits_{|z - p_1^*| < \varepsilon}W_1''(z)} \Big) }.
	\end{align*}
	Consequently,
	\begin{align*}
        &\limsup \limits_{n \to \infty} \frac{1}{b_n} \log \int_U 1_{\{|x| < \varepsilon a_n\}}e^{-nI_1(x/a_n+p_1^*)} R(x/a_n+p_1^*) \, dx 
         \le -(1 - \varepsilon) \inf \limits_{|z - p_1^*| < \varepsilon}W_1''(z) \frac{\gamma^2}{2}.
	\end{align*}
	Using $\log(a+b)\leq \log2+\max\{\log a,\log b\},\, a,b\geq0$, we conclude
	\begin{align*}
        &\limsup \limits_{n \to \infty}\frac{1}{b_n} \log \int_U e^{-nI_1(x/a_n+p_1^*)} R(x/a_n+p_1^*) \, dx \\
        \le{}& \max \bigg\{\limsup \limits_{n \to \infty}\frac{1}{b_n} \log \int_U 1_{\{|x| < \varepsilon a_n\}}e^{-nI_1(x/a_n+p_1^*)} R(x/a_n+p_1^*) \, dx , \\
        &\hphantom{\max \bigg\{} \limsup \limits_{n \to \infty} \frac{1}{b_n}\log \int_U 1_{\{|x| \geq \varepsilon a_n\}}e^{-nI_1(x/a_n+p_1^*)} R(x/a_n+p_1^*) \, dx\bigg\} 
        + \frac{\log2}{b_n} \\ \le{}& -(1 - \varepsilon) \inf \limits_{|z - p_1^*| < \varepsilon}W_1''(z) \frac{\gamma^2}{2}.
	\end{align*}
	Letting $\varepsilon \to 0$ now yields
	\begin{align}
        \limsup \limits_{n \to \infty} \frac{1}{b_n} \log \int_U e^{-nI_1(x/a_n + p_1^*)} R(x/a_n + p_1^*) \, dx \le -W_1''(p_1^*) \frac{\gamma^2}{2}. \label{mdp_upper}
	\end{align}
	For the lower bound \eqref{LDP_open}, let $O \subset \mathbb{R}$ be an arbitrary nonempty open set. Set again $\gamma := \inf \limits_{x\in O} |x| < \infty$. By the definition of $\gamma$ the set $O \cap \{\lv x\rv<\gamma + \varepsilon\}$ is a nonempty open set. Therefore by Taylor's theorem
	\begin{align*}
        &\int_O e^{-nI_1(x/a_n+p_1^*)} R(x/a_n+p_1^*) \, dx \\
        \ge{}& \int_O 1_{\{|x| < \gamma + \varepsilon\}}e^{-nI_1(x/a_n+p_1^*)} R(x/a_n+p_1^*) \, dx \\
        \ge{}& \inf \limits_{|y-p^*| < (\gamma + \varepsilon)/a_n} R(y) \lambda(O \cap \{\lv x\rv<\gamma + \varepsilon\}) \exp\Big({-}n(\gamma+\varepsilon)^2/(2a_n) \sup \limits_{|y - p_1^*| < (\gamma + \varepsilon)/a_n} W_1''(y)\Big)~,
	\end{align*}
	where $\lambda$ is the Lebesgue measure. This yields
	\begin{align*}
	\liminf \limits_{n \to \infty} \frac{1}{b_n} \log \int_O e^{-nI_1(x/a_n+p_1^*)} R(x/a_n+p_1^*) \, dx \ge -W_1''(p_1^*) \frac{(\gamma + \varepsilon)^2}{2}.
	\end{align*}
	Letting $\varepsilon \to 0$ we therefore get
	\begin{align}
	\liminf \limits_{n \to \infty} \frac{1}{b_n} \log \int_O e^{-nI_1(x/a_n+p_1^*)} R(x/a_n+p_1^*) \, dx \ge -W_1''(p_1^*) \frac{\gamma^2}{2}. \label{mdp_lower}
	\end{align}
	Note that the density of $a_n(p_{2i - 1}^{(n)} - p_1^*)$ is
	\begin{align*}
	\frac{1}{c_n} e^{-nI_1(x/a_n+p_1^*)} R(x/a_n+p_1^*)\, dx,
	\end{align*}
	where $c_n$ is the normalization constant. Plugging $U=O = \mathbb{R}$ into \eqref{mdp_upper} and \eqref{mdp_lower} shows $\lim \limits_{n \to \infty} \frac{1}{b_n} \log c_n = 0$. This proves the large deviations principle for $a_n(p_{2i - 1}^{(n)} - p_1^*)$.
	
	Analogously, $a_n(p_{2i} - p_2^*)$ satisfies a large deviation principle with speed $b_n$ and good rate function $W_2''(p_2^*) x^2/2$. Since the canonical moments are independent, we can conclude that the vector
	\begin{align*}
	a_n \big((p_1^{(n)}, \ldots, p_k^{(n)}) - \vec{y}^*\big)
	\end{align*}
	satisfies a large deviations principle with speed $b_n$ and good rate function $\|Hx\|_2^2/2$, where the matrix $H$ is given by $H = \operatorname{diag}(W_1''(p_1^*), W_2''(p_2^*), W_1''(p_1^*), \ldots)^{1/2} \in \mathbb{R}^{k \times k}$. Recall that $\vec{y}^*=(p_1^*,p_2^*,p_1^*,\dots)\in(0,1)^k$. 
	
	In order to transfer this large deviations principle to the sequence of ordinary moments, we need to apply the delta-method for large  deviations. As 
	Theorem 3.1 in \cite{gaozhao2011} states, the sequence
	\begin{align*}
        a_n\big((m_1^{(n)}, \ldots, m_k^{(n)}) - (m_1^*, \ldots, m_k^*)\big) = a_n\big(\varphi_k^{[a,b]}(p_1^{(n)}, \ldots, p_k^{(n)}) - \varphi_k^{[a,b]}(\vec{y}^*)\big)
	\end{align*}
	satisfies a large deviations principle with good rate function
	\begin{align*}
        I(x) := \inf \{\|Hy\|_2^2/2 \mid (D\varphi_k^{[a,b]}(\vec{y}^*))y = x\} =  \|HD\varphi_k^{[a,b]}(\vec{y}^*)^{-1} x\|_2^2/2.
	\end{align*}
\end{proof}

{\bf Acknowledgements.} The authors would like to thank
M. Stein who typed parts of this manuscript with considerable technical
expertise.
The work of  H. Dette and D. Tomecki was supported by the Deutsche Forschungsgemeinschaft (DFG Research Unit 1735, DE 502/26-2, RTG
2131:  High-dimensional Phenomena in Probability - Fluctuations and Discontinuity). 
The work of M. Venker was supported by  the European Research Council under the European Unions Seventh
Framework Programme (FP/2007/2013)/ ERC Grant Agreement n.  307074 as well as by CRC 701 ``Spectral Structures and Topological 
Methods in Mathematics''.
\bigskip

 \bibliographystyle{apalike}
\bibliography{quellen}

\begin{thebibliography}{}

\bibitem[Akemann et~al., 2011]{Handbook}
Akemann, G., Baik, J., and Di~Francesco, P., editors (2011).
\newblock {\em The {O}xford handbook of random matrix theory}.
\newblock Oxford University Press, Oxford.

\bibitem[Anshelevich, 2007]{Anshelevich}
Anshelevich, M. (2007).
\newblock Free {M}eixner states.
\newblock {\em Commun. Math. Phys.}, 276(3):863--899.

\bibitem[Bai and Silverstein, 2010]{BaiSilverstein}
Bai, Z. and Silverstein, J.~W. (2010).
\newblock {\em Spectral analysis of large dimensional random matrices}.
\newblock Springer Series in Statistics. Springer, New York, second edition.

\bibitem[Castro and Gr\"unbaum, 2013]{CastroGrunbaum}
Castro, M.~M. and Gr\"unbaum, F.~A. (2013).
\newblock On a seminal paper by {K}arlin and {M}c{G}regor.
\newblock {\em SIGMA Symmetry Integrability Geom. Methods Appl.}, 9:Paper 020,
  11.

\bibitem[Chang et~al., 1993]{chakemstu1993}
Chang, F.~C., Kemperman, J. H.~B., and Studden, W.~J. (1993).
\newblock A normal limit theorem for moment sequences.
\newblock {\em Ann. Probab.}, 21:1295--1309.

\bibitem[Chihara, 1978]{chihara1978}
Chihara, T.~S. (1978).
\newblock {\em An Introduction to Orthogonal Polynomials}.
\newblock Gordon and Breach, New York.

\bibitem[Cohen and Trenholme, 1984]{CohenTrenholme}
Cohen, J.~M. and Trenholme, A.~R. (1984).
\newblock Orthogonal polynomials with a constant recursion formula and an
  application to harmonic analysis.
\newblock {\em J. Funct. Anal.}, 59(2):175--184.

\bibitem[Dembo and Zeitouni, 2010]{DZ}
Dembo, A. and Zeitouni, O. (2010).
\newblock {\em Large deviations techniques and applications}, volume~38 of {\em
  Stochastic Modelling and Applied Probability}.
\newblock Springer-Verlag, Berlin.
\newblock Corrected reprint of the second (1998) edition.

\bibitem[Dette and Nagel, 2012]{detnag2012}
Dette, H. and Nagel, J. (2012).
\newblock Distributions on unbounded moment spaces and random moment sequences.
\newblock {\em Ann. Probab.}, 40(6):2690--2704.

\bibitem[Dette and Studden, 1997]{dettstud1997}
Dette, H. and Studden, W.~J. (1997).
\newblock {\em Canonical Moments with Applications in Statistics, Probability
  and Analysis}.
\newblock Wiley and Sons, New York.

\bibitem[Dumitriu and Edelman, 2002]{DE}
Dumitriu, I. and Edelman, A. (2002).
\newblock Matrix models for beta ensembles.
\newblock {\em J. Math. Phys.}, 43(11):5830--5847.

\bibitem[Gamboa and Lozada-Chang, 2004]{gamloz2004}
Gamboa, F. and Lozada-Chang, L.~V. (2004).
\newblock Large deviations for random power moment problem.
\newblock {\em Ann. Probab.}, 32:2819--2837.

\bibitem[Gamboa et~al., 2016]{GNR1}
Gamboa, F., Nagel, J., and Rouault, A. (2016).
\newblock Sum rules via large deviations.
\newblock {\em J. Funct. Anal.}, 270(2):509--559.

\bibitem[Gamboa et~al., 2017]{GNR2}
Gamboa, F., Nagel, J., and Rouault, A. (2017).
\newblock Sum rules and large deviations for spectral measures on the unit
  circle.
\newblock {\em Random Matrices Theory Appl.}, 6(1):1750005, 49.

\bibitem[Gao and Zhao, 2011]{gaozhao2011}
Gao, F. and Zhao, X. (2011).
\newblock Delta method in large deviations and moderate deviations for
  estimators.
\newblock {\em Ann. Statist.}, 39(2):1211--1240.

\bibitem[Hamburger, 1920]{hamburger1920}
Hamburger, H. (1920).
\newblock {\"{U}}ber eine {E}rweiterung des {S}tieltjesschen
  {M}omentenproblems.
\newblock {\em Math. Ann.}, 81:235--319.

\bibitem[Hiai and Petz, 2000]{hiaipetz2000}
Hiai, F. and Petz, D. (2000).
\newblock {\em The Semicircle Law, Free Random Variables and Entropy}.
\newblock American Mathematical Society, R.I.

\bibitem[Karlin and Shapeley, 1953]{karsha1953}
Karlin, S. and Shapeley, L.~S. (1953).
\newblock Geometry of moment spaces.
\newblock In {\em Amer. Math. Soc. Memoir No. 12}. American Mathematical
  Society, Providence, Rhode Island.

\bibitem[Karlin and Studden, 1966]{karlin1966}
Karlin, S. and Studden, W. (1966).
\newblock {\em Tchebycheff systems: with applications in analysis and
  statistics}.
\newblock Interscience Publishers.

\bibitem[Kesten, 1959]{Kesten}
Kesten, H. (1959).
\newblock Symmetric random walks on groups.
\newblock {\em Trans. Amer. Math. Soc.}, 92:336--354.

\bibitem[Krein and Nudelman, 1977]{krenud1977}
Krein, M.~G. and Nudelman, A.~A. (1977).
\newblock {\em The Markov Moment Problem and Extremal Problems.}
\newblock American Mathematical Society., Providence, RI.

\bibitem[Krishnapur et~al., 2016]{KRV}
Krishnapur, M., Rider, B., and Vir\'ag, B. (2016).
\newblock Universality of the stochastic {A}iry operator.
\newblock {\em Comm. Pure Appl. Math.}, 69(1):145--199.

\bibitem[Lozada-Chang, 2005]{lozada2005}
Lozada-Chang, L.~V. (2005).
\newblock Large deviations on moment spaces.
\newblock {\em Electron. J. Probab.}, 10:662--690.

\bibitem[McKay, 1981]{McKay}
McKay, B.~D. (1981).
\newblock The expected eigenvalue distribution of a large regular graph.
\newblock {\em Linear Algebra Appl.}, 40:203--216.

\bibitem[Meixner, 1934]{Meixner}
Meixner, J. (1934).
\newblock Orthogonale {P}olynomsysteme mit einer besonderen {G}estalt der
  erzeugenden {F}unktion.
\newblock {\em J. London Math. Soc.}, S1-9(1):6.

\bibitem[Nica and Speicher, 2006]{nicspe2006}
Nica, A. and Speicher, R. (2006).
\newblock {\em Lectures on the combinatorics of free probability}.
\newblock London Mathematical Society Lecture Note Series 335, Cambridge.

\bibitem[Saff and Totik, 1997]{SaffTotik}
Saff, E. and Totik, V. (1997).
\newblock {\em Logarithmic potentials with external fields}, volume 316 of {\em
  Grundlehren der Mathematischen Wissenschaften [Fundamental Principles of
  Mathematical Sciences]}.
\newblock Springer-Verlag, Berlin.
\newblock Appendix B by Thomas Bloom.

\bibitem[Saitoh and Yoshida, 2001]{SaitohYoshida}
Saitoh, N. and Yoshida, H. (2001).
\newblock The infinite divisibility and orthogonal polynomials with a constant
  recursion formula in free probability theory.
\newblock {\em Probab. Math. Statist.}, 21(1, Acta Univ. Wratislav. No.
  2298):159--170.

\bibitem[Shohat and Tamarkin, 1943]{ShohatTamarkin}
Shohat, J.~A. and Tamarkin, J.~D. (1943).
\newblock {\em The {P}roblem of {M}oments}.
\newblock American Mathematical Society Mathematical surveys, vol. I. American
  Mathematical Society, New York.

\bibitem[Skibinsky, 1967]{skibinsky1967}
Skibinsky, M. (1967).
\newblock The range of the {$(n + 1)$}-th moment for distributions on {$[0;
  1]$}.
\newblock {\em J. Appl. Probability}, 4:543--552.

\bibitem[Skibinsky, 1968]{skibinsky1968}
Skibinsky, M. (1968).
\newblock Extreme $n$th moments for distributions on $[0,\,1]$ and the inverse
  of a moment space map.
\newblock {\em J. Appl. Probability}, 5:693--701.

\bibitem[Skibinsky, 1969]{skibinsky1969}
Skibinsky, M. (1969).
\newblock Some striking properties of binomial and beta moments.
\newblock {\em Ann. Math. Stat.}, 40:1753--1764.

\bibitem[Stahl and Totik, 1992]{StahlTotik}
Stahl, H. and Totik, V. (1992).
\newblock {\em General orthogonal polynomials}, volume~43 of {\em Encyclopedia
  of Mathematics and its Applications}.
\newblock Cambridge University Press, Cambridge.

\bibitem[Verblunsky, 1935]{verblunsky1935}
Verblunsky, S. (1935).
\newblock On positive harmonic functions: A contribution to the algebra of
  {F}ourier series.
\newblock {\em Proc. London Math. Soc.}, 38:125--157.

\bibitem[Verblunsky, 1936]{verblunsky1936}
Verblunsky, S. (1936).
\newblock On positive harmonic functions (second paper).
\newblock {\em Proc. London Math. Soc.}, 40:290--320.

\end{thebibliography}

\end{document}